\documentclass{siamltex}
\usepackage{amsmath}
\usepackage{amsfonts}
\usepackage{amssymb}
\usepackage{graphicx}
\usepackage{color}
\usepackage{bm}
\usepackage{caption}
\usepackage{verbatim}
\usepackage{hyperref}
\usepackage{bmpsize}
\allowdisplaybreaks

\newtheorem{algorithm}[theorem]{Algorithm}

\newcommand{\be}{\begin{equation}}
\newcommand{\ee}{\end{equation}}
\newcommand{\bea}{\begin{eqnarray}}
\newcommand{\eea}{\end{eqnarray}}
\newcommand{\beas}{\begin{eqnarray*}}
\newcommand{\eeas}{\end{eqnarray*}}

\newcommand{\vertiii}[1]{{\left\vert\kern-0.25ex\left\vert\kern-0.25ex\left\vert #1 
    \right\vert\kern-0.25ex\right\vert\kern-0.25ex\right\vert}}

\begin{document}
\title{An efficient, partitioned ensemble algorithm for simulating ensembles of evolutionary MHD flows at low magnetic Reynolds number}

\author{
Nan Jiang\thanks{Department of Mathematics and Statistics,  
Missouri University of Science and Technology,
Rolla, MO 65409-0020, {\tt jiangn@mst.edu}. This author was partially supported by the US National Science Foundation grant DMS-1720001 and a University of Missouri Research Board grant}.
\and Michael Schneier \thanks{Department of Scientific Computing, Florida State University, Tallahassee, FL 32306-4120,
 {\tt mschneier89@gmail.com}. This author was supported by the US Air Force Office of Scientific Research grant FA9550-15-1-0001 and US Department of Energy Office of Science grants DE-SC0009324 and DE-SC0010678.
}
}
\maketitle

\begin{abstract}
Studying the propagation of uncertainties in a nonlinear dynamical system usually involves generating a set of samples in the stochastic parameter space and then repeated simulations with different sampled parameters. The main difficulty faced in the process is the excessive computational cost. In this paper, we present an efficient, partitioned ensemble algorithm to determine multiple realizations of a reduced Magnetohydrodynamics (MHD) system, which models MHD flows at low magnetic Reynolds number. The algorithm decouples the fully coupled problem into two smaller sub-physics problems, which reduces the size of the linear systems that to be solved and allows the use of optimized codes for each sub-physics problem. Moreover, the resulting coefficient matrices are the same for all realizations at each time step, which allows faster computation of all realizations and significant savings in computational cost. We prove this algorithm is first order accurate and long time stable under a time step condition. Numerical examples are provided to verify the theoretical results and demonstrate the efficiency of the algorithm.
\end{abstract}

\section{Introduction}

Magnetohydrodynamics (MHD) studies the dynamics of electrically conducting fluids in the presence of a magnetic field. It has many applications in astrophysics, planetary science, plasma physics and metallurgical industries, such as MHD turbulence in accretion disks \cite{Accretion}, geodynamo simulations \cite{geodynamo}, plasma containment in fusion reactors \cite{plasma} and magnetic damping of jets and vortices \cite{damping}. In a typical laboratory or industrial process, liquid-metal MHD usually has a modest conductivity ($\sim 10^6 \,\Omega^{-1}m^{-1}$) and low velocity ($\sim 1m/s$), which makes the induced current densities rather modest. When this modest current density is spread over a small area ($\sim 0.1m$ in a laboratory), the induced magnetic field is usually found to be negligible by comparison with the imposed magnetic field, \cite{Davidson01}. Such flows, i.e. MHD flows that occur at low magnetic Reynolds number, can be modeled by the following reduced MHD system,  \cite{Gunzburger91, Ingram13, Layton2014, Peterson88}. 

Let $\Omega$ be a bounded Lipschitz domain in $R^d$ $(d=3)$. The governing equations of the reduced MHD system are: Given known body force $f(x,t)$ and imposed static magnetic field $B(x)$, find the fluid velocity $u(x,t)$, the pressure $p(x,t)$ and the electric potential $\phi(x,t)$ such that
\begin{equation}\label{eq:MHD}
\left\{\begin{aligned}
&\frac{1}{N} \left( u_{t}+u\cdot\nabla u \right) - \frac{1}{M^{2}}\triangle u+\nabla p  
=f +\left( B\times \nabla \phi + B\times (B\times u)\right),\\
&\Delta \phi = \nabla \cdot (u\times B) \text{ and } \nabla \cdot u =\nabla \cdot B =0, \quad\forall (x,t)\in\Omega\times(0,T],\\
&u=\phi   =0, \quad\forall (x,t)\in\partial\Omega\times(0,T],\\
&u(x,0)   =u^{0}(x), \quad\forall x\in\Omega,\\
&\phi(x,0)    =\phi^{0}(x), \quad\forall x\in\Omega,
\end{aligned}\right.
\end{equation}
\noindent where $M$ is the Hartman number given by $M=\tilde{B}L\sqrt{\frac{\sigma}{\rho \nu}}$ and N is the interaction parameter given by $N=\sigma \tilde{B}^2 \frac{L}{\rho U}$, in which $\tilde{B}$ is the characteristic magnetic field, $\rho$ is the density, $\nu$ is the kinematic viscosity, $\sigma$ is the electrical conductivity, $U$ is a typical velocity of the motion, $L$ is the characteristic length scale.

Nonlinear dynamical systems such as the MHD system are sensitive to small changes in initial conditions, boundary conditions, body forces and many other input parameters. It is important to understand and quantify the limits of predictability of the system, and to develop computational approaches to reduce simulation time and computational cost while preserving a certain degree of accuracy. Most approaches to represent the uncertainties are ensemble based. Specifically, an ensemble of samples are generated to represent possible events, and then individual simulations are run for each sample. These computations are usually very expensive, and even prohibitive, especially if the size of the ensemble is large. Recently a new ensemble algorithm was proposed for fast calculation of an ensemble of the Navier-Stokes equations \cite{JL14}, which constructs linear systems with the same coefficient matrix for all realizations at each time step and thus allows the use of the either direct methods such as the LU factorization or iterative methods such as block CG \cite{FOP95}, block GMRS \cite{GS96} for fast solving the linear systems. In this report, we extend the ensemble algorithm studied in \cite{JL14} to the reduced MHD system.

Herein we consider computing the reduced MHD system $J$ times with different initial conditions and/or body forces. The solution $(u_j, p_j, \phi_j)$ of $j$-th realization, which corresponds to the initial condition $u^0_j(x)$ and body force $f_j(x,t)$, satisfies, for $j=1,2,...,J$,
\begin{equation}\label{eq:MHD-ensemble}
\left\{\begin{aligned}
&\frac{1}{N} \left( u_{j,t}+u_j\cdot\nabla u_j \right) - \frac{1}{M^2} \triangle u_j+\nabla p_j  
=f_j + \left( B \times \nabla \phi_j + B \times (B \times u_j)\right),\\
&\Delta \phi_j = \nabla \cdot (u_j\times B) \text{ and } \nabla \cdot u_j =\nabla \cdot B =0\quad\forall (x, t)\in\Omega\times(0,T],\\
&u_j=B   =0\quad\forall( x, t)\in\partial\Omega\times(0,T],\\
&u_j(x,0)   =u_j^{0}(x)\quad\forall x\in\Omega,\\
&\phi_j(x,0)    =\phi_j^{0}(x)\quad\forall x\in\Omega.
\end{aligned}\right.
\end{equation}

Two aspects need to be considered to construct an efficient ensemble algorithm to solve the above coupled nonlinear system. The first is to use a partitioned method to uncouple the problem into two separate subproblems. This reduces solving a large linear system to solving two much smaller linear systems, which reduces the computational time and memory storage required. Furthermore, uncoupling the system also makes possible the use of highly optimized legacy code for each sub-physics problem, which reduces the main computational complexity. The other aspect is to design an ensemble algorithm for the reduced MHD system such that all ensemble members share one coefficient matrix at each time step. 

To start, we first define the ensemble mean of the velocity $u_j^n$ and the electric potential $\phi_j^n$ respectively
\begin{align}
\bar{u}^n=\frac{1}{J}\sum_{j=1}^J u_j^n \quad\text{ and } \quad\bar{\phi}^n=\frac{1}{J}\sum_{j=1}^J \phi_j^n,
\end{align}
where $u_j^n= u_j(x, t_n)$, $\phi_j^n=\phi_j(x,t_n)$ and $t_n= n\Delta t$ ($n=0,1,2, ...$).

We then propose a first order, partitioned, ensemble algorithm given by

\begin{algorithm}\label{Algoo}
Sub-problem 1: Given $u_{j}^n$ and $\phi_{j}^n$, find $u_{j}^{n+1}$ and $p_{j}^{n+1}$ satisfying
\begin{equation}
\left\{\begin{aligned}
&\frac{1}{N}\left(\frac{u_{j}^{n+1}-u_{j}^n}{\Delta t}\right)+\frac{1}{N}\bar{u}^n\cdot\nabla u_{j}^{n+1}+\frac{1}{N}(u_{j}^n-\bar{u}^n)\cdot\nabla u_{j}^n - \frac{1}{M^2}\Delta u_{j}^{n+1}\\
&\qquad\qquad\qquad\qquad+\nabla p_{j}^{n+1}= f^{n+1}_j+\left( B\times \nabla \phi_j^n + B\times (B\times u_j^{n+1})\right),\nonumber\\
&\nabla \cdot u_{j}^{n+1}=0. \nonumber
\end{aligned}\right. 
\end{equation}

Sub-problem 2: Given $u_{j}^{n}$, find $\phi_{j}^{n+1}$ satisfying

\begin{equation}
\Delta \phi_j^{n+1} = \nabla \cdot (u_j^n\times B) .  \nonumber
\end{equation}

\end{algorithm}

In Sub-problem 1, moving all the known quantities (at time level $t_n$) to the right hand side, one can see all ensemble members $u_j$ have the same coefficient matrix. Sub-problem 2 is a linear problem for $\phi_j$ that results in one common constant coefficient matrix for all realizations. Sub-Problem 1 and 2 are fully uncoupled at each time step and can be run in parallel. Naturally, if the ensemble is large, it can be divided into several subgroups and then one can apply the algorithm to each subgroup.

This paper is organized into four sections. In Section $2$ we establish the notation and give a weak formulation of the reduced MHD system. In Section $3$ we prove the long-time stability of the proposed algorithm under a timestep condition. In Section $4$ we present the convergence analysis of the algorithm. Several numerical examples are presented in Section $5$ to describe the implementation of the algorithm and to demonstrate its efficiency.

\subsection{Previous works on ensemble methods}
The ensemble method was first proposed by Jiang and Layton in \cite{JL14} to efficiently compute ensembles of Navier-Stokes equations with low/modest Reynolds numbers. For high Reynolds number flows, two ensemble eddy viscosity regularization methods were studied in \cite{JL15}, and a time relaxation algorithm in \cite{TNW16} . Higher order ensemble methods can be found in \cite{J15, J17}. To further reduce the computation cost, incorporating reduced order modeling techniques with the ensemble algorithm was investigated in \cite{GJS17, GJS16}. An ensemble algorithm for computing flows with varying model parameters were developed in \cite{GJW16, GJW17}. The ensemble method has also been extended for computing full MHD flows in Els$\ddot{s}$sser variables in \cite{MR17}.

\section{Notation and preliminaries}

Throughout this paper the $L^2(\Omega)$ norm of scalars, vectors, and tensors will be denoted by $\Vert \cdot\Vert$ with the usual $L^2$ inner product denoted by $(\cdot, \cdot)$.  $H^{k}(\Omega)$ is the Sobolev space $%
W_{2}^{k}(\Omega)$, with norm $\Vert\cdot\Vert_{k}$. For functions $v(x,t)$ defined on $(0,T)$, we define the norms, for
$1\leq m<\infty$,
\[
\| v \|_{\infty,k} \text{ }:=EssSup_{[0,T]}\| v(\cdot, t)\|_{k}\qquad \text{and}\qquad \|v\|_{m,k} \text{ }:= \Big(  \int_{0}^{T}\|v(\cdot, t)\|_{k}^{m}\, dt\Big)
^{1/m} \text{ .}%
\]
The function spaces we consider are:
\begin{align*}
X:&=H_{0}^{1}(\Omega )^{d}=\left\lbrace v\in L^2(\Omega)^d: \nabla v\in L^2(\Omega)^{d\times d}\text{ and }v=0 \text{ on } \partial \Omega\right\rbrace,\\
Q : &=L_{0}^{2}(\Omega )=\left\lbrace q\in L^2(\Omega): \int_{\Omega} q \text{ }dx=0\right\rbrace,\\
S : &=H_{0}^{1}(\Omega )=\left\lbrace \phi\in L^2(\Omega): \nabla \phi\in L^2(\Omega)\text{ and }\phi=0 \text{ on } \partial \Omega\right\rbrace,\\
V : &=\left\lbrace v\in X:  (\nabla\cdot v, q)=0, \forall q\in Q \right\rbrace.
\end{align*}

The norm on the dual space of $X$ is defined by
\begin{equation*}
\Vert f\Vert _{-1}=\sup_{0\neq v\in X}\frac{(f,v)}{\Vert
\nabla v\Vert }\text{ .}
\end{equation*}

A weak formulation of the reduced MHD equations is: Find $u: [0,T]\rightarrow X$, $p: [0,T]\rightarrow Q$, and $\phi :  [0,T] \rightarrow S$ for a.e. $t\in (0,T]$ satisfying
\begin{align}
&\frac{1}{N} \left( u_{j,t}, v\right)+\frac{1}{N}\left( u_j\cdot\nabla u_j, v \right) +\frac{1}{M^2} \left(\nabla  u_j, \nabla v\right)- \left( p_j, \nabla \cdot v\right)\\  
&\qquad\qquad\qquad\qquad\qquad+\left( -\phi_j + u_j\times B, v\times B\right)=\left( f_j, v\right), \quad \forall v\in X, \nonumber\\
&\left( \nabla \cdot u_j, q\right)=0, \quad \forall q\in Q,\nonumber\\
&-\left( \nabla \phi_j, \nabla \psi\right) +\left(  u_j\times B, \nabla \psi\right)=0,\quad \forall \psi \in S. \nonumber
\end{align}

We will use the discrete
Gronwall inequality (Lemma \ref{lm:Gronwall} below) in the error analysis, see \cite{HR90} for proof.

\begin{lemma}\label{lm:Gronwall}
 Let $ D \geq 0$ and $\kappa_n, A_n,B_n, C_n \geq 0$ for
any integer $n\geq 0$ and satisfy

$$A_{\tilde{N}}+\Delta t \sum_{n=0}^{\tilde{N}}B_n\leq \Delta t \sum_{n=0}^{\tilde{N}}\kappa_n A_n+\Delta t \sum_{n=0}^{\tilde{N}} C_n+D \text{ for } \tilde{N}\geq 0.$$

Suppose that for all n, $\Delta t \kappa_n \leq 1,$ and set
$g_n=(1-\Delta t \kappa_n)^{-1}$. Then,

$$A_{\tilde{N}}+\Delta t \sum_{n=0}^{\tilde{N}}B_n \leq exp(\Delta t \sum_{n=0}^{\tilde{N}} g_n\kappa_n)[\Delta t \sum_{n=0}^{\tilde{N}} C_n+D] \text{ for } \tilde{N}\geq 0. $$

\end{lemma}
We denote conforming velocity, pressure, potential finite element spaces based on an edge to edge
triangulation ($d=2$) or tetrahedralization ($d=3$) of $\Omega $ with maximum element diameter $h$ by 
\begin{equation*}
X_{h}\subset X\text{ }\text{, }Q_{h}\subset
Q\text{, }S_{h}\subset
S.
\end{equation*}%
We also assume the finite element spaces ($X_{h}$, $Q_{h}$) satisfy the usual discrete inf-sup /$ LBB^{h}$
condition for stability of the discrete pressure, see \cite{G89} for more on this condition. Taylor-Hood elements, e.g., \cite{BS08}, \cite{G89}, are one such choice used in the tests in Section $5$. We further assume the finite element spaces satisfy the approximation properties of piecewise polynomials on quasiuniform meshes
\begin{align}
\inf_{v_h\in X_h}\| v- v_h \|&\leq C h^{k+1}\Vert u \Vert_{k+1}	   &\forall v\in [H^{k+1}(\Omega)]^d, \label{Interp1}\\
\inf_{v_h\in X_h}\| \nabla ( v- v_h )\|&\leq C h^k \Vert v\Vert_{k+1}&\forall v\in [H^{k+1}(\Omega)]^d, \label{interp2}\\
\inf_{q_h\in Q_h}\| q- q_h \|&\leq C h^{s+1}\Vert p\Vert_{s+1}	   &\forall q\in H^{s+1}(\Omega),      \label{interp3}\\
\inf_{\psi_h\in S_h}\| \psi- \psi_h \|&\leq C h^{m+1}\Vert \psi \Vert_{m+1}	   &\forall \psi\in H^{m+1}(\Omega), \label{Interp4}\\
\inf_{\psi_h\in S_h}\| \nabla ( \psi- \psi_h )\|&\leq C h^m \Vert \psi\Vert_{m+1}&\forall \psi\in H^{m+1}(\Omega), \label{interp5}
\end{align}

\noindent where the generic constant $C>0$ is independent of mesh size $h$.  
An example for which the $LBB_h$ stability condition and the approximation properties are satisfied is the finite elements pair ($P^{k+1}$--$P^{k}$--$P^{k+1}$), $k\geq 1$.
For finite element methods see
\cite{GR79,GR86, G89,Layton08} for more details.

The discretely divergence free
subspace of $X_{h}$ is 
\begin{equation*}
V_{h} :\text{}=\{v_{h}\in X_{h}:(\nabla \cdot v_{h},q_{h})=0\text{ , }%
\forall q_{h}\in Q_{h}\}.
\end{equation*}
We assume the mesh and finite element spaces satisfy the standard inverse inequality 
\begin{equation}\label{inverse}
h\Vert\nabla v_{h}\Vert   \leq C_{(inv)}\Vert v_{h}\Vert. 
\end{equation}
that is known to hold for standard finite element spaces with locally quasi-uniform meshes \cite{BS08}. 
We also define the standard explicitly skew-symmetric trilinear form
\[
b^{\ast}(u,v,w):=\frac{1}{2}(u\cdot\nabla v,w)-\frac{1}{2}(u\cdot\nabla w,v) 
\]
that satisfies the bound \cite{Layton08}
\begin{gather}
b^{\ast}(u,v,w)\leq C \Vert \nabla u\Vert\Vert\nabla v\Vert\Vert\nabla
w \Vert, \quad \forall\, u, v, w \in X ,\label{In1}\\
b^{\ast}(u,v,w)\leq C \Vert \nabla u\Vert\Vert\nabla v\Vert\left(\Vert\nabla
w \Vert\Vert w\Vert\right)^{1/2}, \quad \forall\, u, v, w \in X ,\label{In2}\\
b^{\ast}(u,v,w)\leq C (\Vert \nabla u\Vert\Vert  u\Vert)^{1/2}\Vert\nabla v\Vert\Vert\nabla
w \Vert, \quad \forall\, u, v, w \in X .\label{In3}
\end{gather}

The full discretization of the proposed partitioned ensemble algorithm is
\begin{algorithm}\label{Algo}
Sub-problem 1: Given $u_{j,h}^n \in X_h$ and $\phi_{j,h}^n \in S_h$, find $u_{j,h}^{n+1}\in X_h$ and $p_{j,h}^{n+1}\in Q_h$ satisfying
\begin{equation}\label{1}
\left\{\begin{aligned}
&\frac{1}{N}\left(\frac{u_{j,h}^{n+1}-u_{j,h}^n}{\Delta t}, v_h\right)+\frac{1}{N}b^{\ast}(\bar{u}_h^n, u_{j,h}^{n+1}, v_h)+\frac{1}{N}b^{\ast}(u_{j,h}^n-\bar{u}_h^n, u_{j,h}^n,v_h) \\
&+\frac{1}{M^2}(\nabla u_{j,h}^{n+1}, \nabla v_h)-(p_{j,h}^{n+1}, \nabla \cdot v_h)+\left(-\nabla \phi^n_{j,h}+u_{j,h}^{n+1}\times B, v_h\times B\right)\\
&=\left( f^{n+1}_j, v_h\right), \qquad\forall v_h\in X_h, \\
&\left( \nabla \cdot u_{j,h}^{n+1}, q_h\right)=0, \qquad \forall q_h\in Q_h.
\end{aligned}\right. 
\end{equation}

Sub-problem 2: Given $u_{j,h}^{n} \in X_h$, find $\phi_{j,h}^{n+1}\in S_h$ satisfying
\begin{equation}\label{2}
\left(-\nabla \phi_{j,h}^{n+1}+u_{j,h}^n\times B, \nabla \psi_h\right) =0, \qquad  \forall \psi_h\in S_h.
\end{equation}

\end{algorithm}
\section{Stability of the method}
In this section, we prove Algorithm \eqref{Algo} is long time, nonlinearly stable under a CFL like time step condition.
\begin{theorem}
[Stability]\label{th:stability} Consider the method with a standard
spacial discretization with mesh size $h$. Suppose the following time step conditions hold
\begin{align}
C \frac{M^2}{N}\frac{\Delta t}{h}
\Vert \nabla (u_{j,h}^{ n}-\bar{u}_h^n)\Vert^2\leq 1,  \qquad j= 1, ..., J,\label{ineq:CFL} 
\end{align}
then, for
any
$n\geq 1$
\begin{align}\label{qq7}
&\frac{1}{N}\Vert u_{j,h}^{n}\Vert^{2}+\sum_{k=0}^{n-1}\frac{1}{2N}\Vert u_{j,h}^{k+1}-u_{j,h}^{k}\Vert^2+\sum_{k=0}^{n-1}\frac{\Delta t}{2M^2}\Vert \nabla u_{j,h}^{k+1}\Vert^2+\frac{\Delta t}{2M^2} \Vert\nabla u_{j,h}^{n}\Vert^2 \\
&+ \Delta t\Vert B\times u_{j,h}^{n}\Vert^2
+\Delta t\sum_{k=0}^{n-1}\left(\Vert - \phi^k_{j,h} +u_{j,h}^{k+1}\times B\Vert^2+\Vert - \phi^{k+1}_{j,h} +u_{j,h}^{k}\times B\Vert^2\right)\nonumber\\
& +\Delta t \Vert \nabla \phi_{j,h}^{n} \Vert^2\leq \frac{1}{N}\Vert u_{j,h}^{0}\Vert^{2}+\frac{\Delta t}{2M^2} \Vert\nabla u_{j,h}^{0}\Vert^2+ \Delta t\Vert B\times u_{j,h}^{0}\Vert^2 \nonumber\\
&+\Delta t \Vert \nabla \phi_{j,h}^{0} \Vert^2+\Delta t\sum_{k=0}^{n-1}M^2\Vert f_j^{k+1}\Vert_{-1}^2\text{ .}\nonumber
\end{align}

\end{theorem}

\begin{proof}
Set $v_h=u_{j,h}^{n+1}$ in (\ref{1}) and multiply through by $\Delta t$. This gives
\begin{align}\label{qq1}
&\frac{1}{2N}\left(\Vert u_{j,h}^{n+1}\Vert^{2}-\Vert u_{j,h}^{n}\Vert^{2}+\Vert u_{j,h}^{n+1}-u_{j,h}^{n}\Vert^2\right)
+ \frac{\Delta t}{N} b^*(  u_{j,h}^n-\bar{u}_h^n, u_{j,h}^n,u_{j,h}^{n+1})\\
&+\frac{\Delta t}{M^2} \Vert\nabla u_{j,h}^{n+1}\Vert^2+ \Delta t\Vert B\times u_{j,h}^{n+1}\Vert^2
-\Delta t\left(\nabla \phi_{j,h}^n,u_{j,h}^{n+1}\times B\right)= \Delta t(f_j^{n+1}, u_{j,h}^{n+1})\text{ .}\nonumber
\end{align}

\noindent Set $\psi_h=\phi_{j,h}^{n+1}$ in (\ref{2}) and multiply through by $\Delta t$. This gives
\begin{gather}\label{qq2}
\Delta t\Vert \nabla \phi_{j,h}^{n+1} \Vert^2=\Delta t ( \nabla \phi_{j,h}^{n+1}, u_{j,h}^n\times B).
\end{gather}

\noindent The following equality will be used in the next step.

\begin{align}\label{qqq2}
& \Vert B\times u_{j,h}^{n+1}\Vert^2+\Vert \nabla \phi_{j,h}^{n+1} \Vert^2
+\left(-\nabla \phi_{j,h}^n,u_{j,h}^{n+1}\times B\right)+ ( -\nabla \phi_{j,h}^{n+1}, u_{j,h}^n\times B)\nonumber\\
&=\Vert B\times u_{j,h}^{n+1}\Vert^2+\Vert \nabla \phi_{j,h}^{n+1} \Vert^2\nonumber\\
&\qquad+\frac{1}{2}\left(\Vert  -\nabla \phi_{j,h}^{n} + u_{j,h}^{n+1}\times B\Vert ^2 -\Vert \nabla \phi_{j,h}^{n} \Vert^2-\Vert u_{j,h}^{n+1}\times B\Vert^2\right)\\
&\qquad+\frac{1}{2}\left(\Vert  -\nabla \phi_{j,h}^{n+1} + u_{j,h}^{n}\times B\Vert ^2 -\Vert \nabla \phi_{j,h}^{n+1} \Vert^2-\Vert u_{j,h}^{n}\times B\Vert^2\right)\nonumber\\
&=\frac{ 1}{2}\left(\Vert B\times u_{j,h}^{n+1}\Vert^2-\Vert B\times u_{j,h}^{n}\Vert^2\right)+\frac{1}{2}\left(\Vert \nabla \phi_{j,h}^{n+1} \Vert^2- \Vert \nabla \phi_{j,h}^{n} \Vert^2\right)\nonumber\\
&\qquad+\frac{1}{2}\left(\Vert - \phi^n_{j,h} +u_{j,h}^{n+1}\times B\Vert^2+\Vert - \phi^{n+1}_{j,h} +u_{j,h}^{n}\times B\Vert^2\right) .\nonumber
\end{align}
\noindent Adding \eqref{qq1} and \eqref{qq2} and using equality \eqref{qqq2} gives
\begin{align}\label{qq3}
&\frac{1}{2N}\left(\Vert u_{j,h}^{n+1}\Vert^{2}-\Vert u_{j,h}^{n}\Vert^{2}+\Vert u_{j,h}^{n+1}-u_{j,h}^{n}\Vert^2\right)
+ \frac{\Delta t}{N} b^*(  u_{j,h}^n-\bar{u}_h^n, u_{j,h}^n,u_{j,h}^{n+1})\\
&+\frac{\Delta t}{M^2} \Vert\nabla u_{j,h}^{n+1}\Vert^2+\frac{ \Delta t}{2}\left(\Vert B\times u_{j,h}^{n+1}\Vert^2-\Vert B\times u_{j,h}^{n}\Vert^2\right)+\frac{\Delta t}{2}\left(\Vert \nabla \phi_{j,h}^{n+1} \Vert^2- \Vert \nabla \phi_{j,h}^{n} \Vert^2\right)\nonumber\\
&+\frac{\Delta t}{2}\left(\Vert - \phi^n_{j,h} +u_{j,h}^{n+1}\times B\Vert^2+\Vert - \phi^{n+1}_{j,h} +u_{j,h}^{n}\times B\Vert^2\right) =\Delta t(f_j^{n+1}, u_{j,h}^{n+1})\text{ .}\nonumber
\end{align}

\noindent Applying Cauchy-Schwarz and Young's inequality on the right hand side of the equation gives 
\begin{align}\label{qqq3}
&\frac{1}{2N}\left(\Vert u_{j,h}^{n+1}\Vert^{2}-\Vert u_{j,h}^{n}\Vert^{2}+\Vert u_{j,h}^{n+1}-u_{j,h}^{n}\Vert^2\right)+\frac{\Delta t}{M^2} \Vert\nabla u_{j,h}^{n+1}\Vert^2
\\
&+\frac{ \Delta t}{2}\left(\Vert B\times u_{j,h}^{n+1}\Vert^2-\Vert B\times u_{j,h}^{n}\Vert^2\right)+\frac{\Delta t}{2}\left(\Vert \nabla \phi_{j,h}^{n+1} \Vert^2- \Vert \nabla \phi_{j,h}^{n} \Vert^2\right)\nonumber\\
&+\frac{\Delta t}{2}\left(\Vert - \phi^n_{j,h} +u_{j,h}^{n+1}\times B\Vert^2+\Vert - \phi^{n+1}_{j,h} +u_{j,h}^{n}\times B\Vert^2\right)\nonumber\\
& \leq - \frac{\Delta t}{N} b(  u_{j,h}^n-\bar{u}_h^n, u_{j,h}^n,u_{j,h}^{n+1})+\frac{\Delta t}{2M^2}\Vert \nabla u_{j,h}^{n+1}\Vert^2+\frac{M^2\Delta t}{2}\Vert f_j^{n+1}\Vert_{-1}^2\text{ .}\nonumber
\end{align}

\noindent Next, we bound the trilinear terms using \eqref{In2}, \eqref{inverse} and Young's inequality.
\begin{align}
&- \frac{\Delta t}{N} b^{*}( u_{j,h}^{ n}-\bar{u}_h^n,  u_{j,h}^{n},u_{j,h}^{n+1})\\
&=-\frac{\Delta t}{N} b^{*}( u_{j,h}^{ n}-\bar{u}_h^n,  u_{j,h}^{n},u_{j,h}^{n+1}-u_{j,h}^n)\nonumber\\
&\leq C\frac{\Delta t}{N}\Vert\nabla (u_{j,h}^{n}-\bar{u}_h^n)\Vert\Vert\nabla u_{j,h}^{n}\Vert\Vert \nabla (u_{j,h}^{n+1}-u_{j,h}^n)\Vert^{\frac{1}{2}}\Vert u_{j,h}^{n+1}-u_{j,h}^n\Vert^{\frac{1}{2}}\nonumber\\
&\leq C\frac{\Delta t}{N} h^{-\frac{1}{2}
}\Vert\nabla (u_{j,h}^{n}-\bar{u}_h^n)\Vert\Vert\nabla u_{j,h}^{n}\Vert\Vert u_{j,h}^{n+1}-u_{j,h}^n\Vert\nonumber\\
&\leq  C \frac{\Delta t^{2}}{Nh}
\Vert\nabla (u_{j,h}^{n}-\bar{u}_h^n)\Vert^{2} \Vert\nabla u_{j,h}^{n}\Vert^{2}
+
\frac{1}{4N}\Vert u_{j,h}^{n+1}-u_{j,h}^n\Vert^{2}\text{ .}\nonumber
\end{align}

\noindent With this bound, combining like terms, \eqref{qqq3} becomes,

\begin{align}\label{qq4}
&\frac{1}{2N}\left(\Vert u_{j,h}^{n+1}\Vert^{2}-\Vert u_{j,h}^{n}\Vert^{2}\right)+\frac{1}{4N}\Vert u_{j,h}^{n+1}-u_{j,h}^{n}\Vert^2+\frac{\Delta t}{4M^2} \Vert\nabla u_{j,h}^{n+1}\Vert^2\\
&+\frac{\Delta t}{4M^2} \left(\Vert\nabla u_{j,h}^{n+1}\Vert^2-\Vert\nabla u_{j,h}^{n}\Vert^2\right)
+\frac{\Delta t}{4M^2} \left(1-C\frac{M^2}{N}\frac{\Delta t}{h}\Vert \nabla (u_{j,h}^n-\bar{u}_h^n)\Vert^2\right)\Vert\nabla u_{j,h}^{n}\Vert^2 \nonumber\\
&+\frac{ \Delta t}{2}\left(\Vert B\times u_{j,h}^{n+1}\Vert^2-\Vert B\times u_{j,h}^{n}\Vert^2\right) 
+\frac{\Delta t}{2}\left(\Vert \nabla \phi_{j,h}^{n+1} \Vert^2- \Vert \nabla \phi_{j,h}^{n} \Vert^2\right)\nonumber\\
&+\frac{\Delta t}{2}\left(\Vert - \phi^n_{j,h} +u_{j,h}^{n+1}\times B\Vert^2+\Vert - \phi^{n+1}_{j,h} +u_{j,h}^{n}\times B\Vert^2\right)\nonumber\\
& \leq \frac{M^2\Delta t}{2}\Vert f_j^{n+1}\Vert_{-1}^2\text{ .}\nonumber
\end{align}

\noindent With the time step restriction (\ref{ineq:CFL}) assumed, we have

\begin{align}
1-C\frac{M^2}{N}\frac{\Delta t}{h}\Vert \nabla (u_{j,h}^n-\bar{u}_h^n)\Vert^2\geq 0.\nonumber
\end{align}

\noindent Inequality (\ref{qq4}) then reduces to

\begin{align}\label{qq5}
&\frac{1}{2N}\left(\Vert u_{j,h}^{n+1}\Vert^{2}-\Vert u_{j,h}^{n}\Vert^{2}\right)+\frac{1}{4N}\Vert u_{j,h}^{n+1}-u_{j,h}^{n}\Vert^2+\frac{\Delta t}{4M^2} \Vert\nabla u_{j,h}^{n+1}\Vert^2\\
&+\frac{\Delta t}{4M^2} \left(\Vert\nabla u_{j,h}^{n+1}\Vert^2-\Vert\nabla u_{j,h}^{n}\Vert^2\right)
 +\frac{ \Delta t}{2}\left(\Vert B\times u_{j,h}^{n+1}\Vert^2-\Vert B\times u_{j,h}^{n}\Vert^2\right) \nonumber\\
 &+\frac{\Delta t}{2}\left(\Vert - \phi^n_{j,h} +u_{j,h}^{n+1}\times B\Vert^2+\Vert - \phi^{n+1}_{j,h} +u_{j,h}^{n}\times B\Vert^2\right)\nonumber\\
 &+\frac{\Delta t}{2}\left(\Vert \nabla \phi_{j,h}^{n+1} \Vert^2- \Vert \nabla \phi_{j,h}^{n} \Vert^2\right)
\leq \frac{M^2\Delta t}{2}\Vert f_j^{n+1}\Vert_{-1}^2\text{ .}\nonumber
\end{align}

\noindent Summing up \eqref{qq5} and multiplying through by 2 gives

\begin{align}\label{qq6}
&\frac{1}{N}\Vert u_{j,h}^{n}\Vert^{2}+\sum_{k=0}^{n-1}\frac{1}{2N}\Vert u_{j,h}^{k+1}-u_{j,h}^{k}\Vert^2+\sum_{k=0}^{n-1}\frac{\Delta t}{2M^2}\Vert \nabla u_{j,h}^{k+1}\Vert^2+\frac{\Delta t}{2M^2} \Vert\nabla u_{j,h}^{n}\Vert^2 \\
&+ \Delta t\Vert B\times u_{j,h}^{n}\Vert^2
+\Delta t\sum_{k=0}^{n-1}\left(\Vert - \phi^k_{j,h} +u_{j,h}^{k+1}\times B\Vert^2+\Vert - \phi^{k+1}_{j,h} +u_{j,h}^{k}\times B\Vert^2\right)\nonumber\\
& +\Delta t \Vert \nabla \phi_{j,h}^{n} \Vert^2\leq \frac{1}{N}\Vert u_{j,h}^{0}\Vert^{2}+\frac{\Delta t}{2M^2} \Vert\nabla u_{j,h}^{0}\Vert^2+ \Delta t\Vert B\times u_{j,h}^{0}\Vert^2 \nonumber\\
&+\Delta t \Vert \nabla \phi_{j,h}^{0} \Vert^2+\Delta t\sum_{k=0}^{n-1}M^2\Vert f_j^{k+1}\Vert_{-1}^2\text{ .}\nonumber
\end{align}
\end{proof}

\section{Error Analysis}\label{err_analysis}
{\allowdisplaybreaks
In this section, we give a detailed error analysis of the proposed method under the same type of time-step condition (with possibly different constant $C$ in the condition). Assuming that
$X_{h}$ and $Q_{h}$ satisfy the $LBB^{h}$ condition, Sub-problem 1 in Algorithm \eqref{Algo}
is equivalent to: Given $u_{j,h}^{n}\in V_{h}$ and $\phi_{j,h}^n \in S_h$, for $n=0,1,\ldots, \tilde{N}-1$, find $u_{j,h}%
^{n+1}\in V_{h}$ such that 
\begin{align}\label{eq: conv}
&\frac{1}{N}\left(\frac{u_{j,h}^{n+1}-u_{j,h}^n}{\Delta t}, v_h\right)+\frac{1}{N}b^{\ast}(\bar{u}_h^n, u_{j,h}^{n+1}, v_h)+\frac{1}{N}b^{\ast}(u_{j,h}^n-\bar{u}_h^n, u_{j,h}^n,v_h) \\
&+\frac{1}{M^2}(\nabla u_{j,h}^{n+1}, \nabla v_h)+\left(-\nabla \phi^n_{j,h}+u_{j,h}^{n+1}\times B, v_h\times B\right)=\left( f^{n+1}_j, v_h\right)\quad\forall
v_{h}\in V_{h}.\nonumber
\end{align}

We define the discrete norms as
\begin{align*}
\vertiii{v}_{\infty,k}=\max\limits_{0\leq n\leq \tilde{N}}\Vert v^{n}\Vert_{k} \qquad \text{and}\qquad
\vertiii{v}_{m,k}:= \Big(\sum_{n=0}^{\tilde{N}}||v^{n}||_{k}^{m}\Delta t\Big)^{1/m},
\end{align*}
where $v^n=v(t_n)$ and $t_n=n\Delta t$.

To analyze the rate of convergence of the approximation, we assume that the following regularity for the exact solutions:
\begin{gather*}
u_{j} \in L^{\infty}(0,T;H^{k+1}(\Omega))\cap H^{1}(0,T;H^{k+1}(\Omega))\cap
H^{2}(0,T;L^{2}(\Omega)),\\
p_{j} \in L^{2}(0,T;H^{s+1}(\Omega)), \quad \text{and}\quad f_{j} \in L^{2}
(0,T;L^{2}(\Omega)),\\
\phi_{j} \in L^{\infty}(0,T;H^{m+1}(\Omega))\cap H^{1}(0,T;H^{1}(\Omega)).
\end{gather*}
Let $e_{u,j}^{n}=u_{j}^{n}-u_{j,h}^{n}$ and $e_{\phi,j}^{n}=\phi_{j}^{n}-\phi_{j,h}^{n}$ denote the approximation error of the $j$-th simulation at the time instance $t_n$. We then have the following error estimates.

\begin{theorem}[Convergence of Algorithm \ref{Algo}]\label{th:errBEFE-Ensemble} 
For all $j= 1, \ldots, J$, if the following time step conditions hold
\begin{align}
&C \frac{M^2}{N}\frac{\Delta t}{h}
\Vert \nabla (u_{j,h}^{ n}-\bar{u}_h^n)\Vert^2\leq 1,  \qquad j= 1, ..., J,\label{conv1}\\
&\Delta t
<\left(\frac{CM^6}{N^3 }+16N \Vert B\Vert_{L^{\infty}}^2\right)^{-1},\label{conv2}
\end{align}
then,  there exists a positive constant $C$ independent of the time step such that 
\begin{align}\label{ineq:err00}
&\Vert e_{u,j}^{n}\Vert^{2}+\frac{1}{2}\sum_{l=0}^{n-1}\|e_{u,j}^{l+1}-e_{u,j}^{l}\|^{2}+\Delta t\sum_{l=0}^{n-1}
\frac{N}{2M^2}\Vert\nabla e_{u,j}^{l+1}\Vert^{2}\\
&\qquad +N\Delta t\Vert e_{u,j}^{n}\times B\Vert^2+N\Delta t\Vert \nabla e_{\phi, j}^{n}\Vert^2 
+N\Delta t\sum_{l=0}^{n-1}\Vert -\nabla e_{\phi, j}^{l}+e_{u,j}^{l+1}\times B \Vert^2\nonumber\\
&\qquad+N\Delta t\sum_{l=0}^{n-1}\Vert -\nabla e_{\phi, j}^{l+1}+e_{u,j}^{l}\times B \Vert^2+N\Delta t\sum_{l=0}^{n-1}\Vert \nabla e_{\phi, j}^{l+1}\Vert^2 
+\frac{7N}{12M^2}\Delta t
\Vert\nabla e_{u, j}^{n}\Vert^{2}\nonumber\\
&\leq e^{\frac{T\tilde{C}}{1-\Delta t \tilde{C}}}
\bigg\{\Delta t\left(\frac{CM^6}{N^3 }+13N \Vert B\Vert_{L^{\infty}}^2\right)\| e^{0}_{u,j} \|^{2}+  \Vert e
_{u, j}^{0}\Vert^{2}+N\Delta t\Vert \nabla e_{\phi, j}^{0}\Vert^2\nonumber\\
&+\frac{7N}{12M^2}\Delta t
\Vert\nabla e_{u,j}^{0}\Vert^{2}+C\left(\frac{CM^6}{N^3 }+12N \Vert B\Vert_{L^{\infty}}^2\right)h^{2k+2}\Delta t\| u^{0}_{j} \|^{2}_{k+1}+  Ch^{2k+2}\Vert u
_{j}^{0}\Vert^{2}_{k+1}\nonumber\\
&+CN\Vert B\Vert^2_{\infty}h^{2k+2}\Delta t\Vert u_{j}^{0}\Vert_{k+1}^2+CNh^{m}\Delta t\Vert \phi_{j}^{0}\Vert_{m+1}^2+C\frac{N}{2M^2}h^{2k}\Delta t
\Vert u_{j}^{0}\Vert_{k+1}^{2}\nonumber\\
&+C\frac{N}{M^2} h^{2k}\vertiii{ u_j }^2_{2,k+1}+C h^{2k+1}\Delta t^{-1}\vertiii{ u_j }^2_{2,k+1}+ Ch\Delta t\vertiii{ \nabla u_{j,t}}^2_{2,0}
\nonumber\\
&+C\frac{M^2}{N} h^{2k}\vertiii{ u_j }^2_{2,k+1}+C\frac{M^2}{N} h^{2k}\vertiii{ u_j}^4_{4, k+1}+C\frac{M^3}{N} h^{2k}+C\frac{M^2}{N}\Delta t^2\vertiii{ \nabla u_{j,t}}^2_{2,0}\nonumber\\
&+CM^2N h^{2s+2}\vertiii{p_{j}}_{2,s+1}^{2}+C\frac{M^2}{N}h^{2k+2}\vertiii{u_{j,t}}_{2,k+1}^{2} +C\frac{M^2}{N}\Delta t^{2}\vertiii{u_{j,tt}}_{2,0}^{2}\nonumber\\
& +C N h^{2m}\vertiii{ \phi_j }^2_{2,m+1}+CN\Vert B\Vert_{L^{\infty}}^2 h^{2k+2}\vertiii{ u_j }^2_{2,k+1}+ CN\Delta t^2\vertiii{ \nabla \phi_{j,t}}^2_{2,0}\nonumber\\
&+C Nh^{2m+2}\vertiii{ \phi_j }^2_{2,m+1}+CN\Vert B\Vert_{L^{\infty}}^2\Delta t^2\vertiii{ u_{j,t}}^2_{2,0}
\bigg\}+Ch^{2k+2}\vertiii{u_j}^{2}_{\infty, k+1}\nonumber\\
&+Ch^{2k+2}\Delta t \vertiii{u_{j,t}}_{2, k+1}+C
\frac{N}{M^2}h^{2k}\vertiii{u_{j}}_{2, k+1}+CN\Vert B\Vert_{L^{\infty}}^2 h^{2k+2}\Delta t\vertiii{ u_{j}}_{\infty, k+1}\nonumber\\
&+CNh^{2m}\Delta t\vertiii{  \phi_{j}}_{\infty, m+1} +CNh^{2m} \vertiii{\phi_j}_{2,m+1}+CN\Vert B\Vert_{L^{\infty}}^2h^{2k+2}\vertiii{u_j}_{2, k+1}\nonumber\\
&+CNh^{2m}\vertiii{ \phi_{j}}_{2, m+1} 
+C\frac{N}{M}h^{2k}\Delta t
\vertiii{ u_{j}}_{\infty, k+1}.\nonumber
\end{align}
\end{theorem}

In particular, if Taylor-Hood elements ($k=2$, $s=1$) are used, i.e., the $C^{0}$ piecewise-quadratic velocity space $X_{h}$ and the $C^{0}$ piecewise-linear pressure space
$Q_{h}$, and  $P_2$ element ($m=2$) is used for $S_h$, we then have the following estimate.
\begin{corollary}
Assume that $\Vert e_{u,j}^{0}\Vert$, $\Vert\nabla e_{u,j}^0\Vert$ and , $\Vert\nabla e_{\phi,j}^0\Vert$ are all $O(h^2)$ accurate or better. Then, if $(X_{h},Q_{h}, S_h)$ is chosen as the $(P_2, P_1, P_2)$ elements, we have
\begin{align}
&\Vert e_{u,j}^{n}\Vert^{2}+\Delta t\sum_{l=0}^{n-1}
\frac{N}{2M^2}\Vert\nabla e_{u,j}^{l+1}\Vert^{2}
+N\Delta t\Vert e_{u,j}^{n}\times B\Vert^2+N\Delta t\Vert \nabla e_{\phi, j}^{n}\Vert^2 \\
&+N\Delta t\sum_{l=0}^{n-1}\Vert -\nabla e_{\phi, j}^{l}+e_{u,j}^{l+1}\times B \Vert^2+N\Delta t\sum_{l=0}^{n-1}\Vert -\nabla e_{\phi, j}^{l+1}+e_{u,j}^{l}\times B \Vert^2\nonumber\\
&+N\Delta t\sum_{l=0}^{n-1}\Vert \nabla e_{\phi, j}^{l+1}\Vert^2 
+\frac{7N}{12M^2}\Delta t
\Vert\nabla e_{u, j}^{n}\Vert^{2} \leq C (h^4+\frac{h^5}{\Delta t} + \Delta t^2 +h \Delta t)\text{ .}
\end{align}

\end{corollary}

\begin{proof}
The true solution $(u_{j}, p_j, \phi_j)$ of the reduced MHD system \eqref{eq:MHD-ensemble} satisfies
\begin{align}\label{eq:convtrue1}
\frac{1}{N}\Big(&\frac{u_{j}^{n+1}-u_{j}^{n}}{\Delta t}, v_{h}\Big)+ \frac{1}{N} b^{*}(u_{j}^{n+1},
u_{j}^{n+1}, v_{h})+ \frac{1}{M^2}(\nabla u_{j}^{n+1}, \nabla v_{h})\\
&- (p_{j}%
^{n+1},\nabla\cdot v_{h})+(u_j^{n+1}\times B, v_h\times B_j)-(\nabla \phi_j^n, v_h\times B)\nonumber\\
&=(f_{j}^{n+1}, v_{h})+\left(\nabla (\phi_j^{n+1}-\phi_j^n), v_h\times B\right)
+ \text{Intp}(u_{j}^{n+1};v_{h})\text{ , }\quad\forall v_{h}\in V_{h},\nonumber
\end{align}
\noindent and
\begin{align}\label{eq:convtrue2}
&\left(-\nabla \phi_j^{n+1}+u_j^n\times B, \nabla \psi_h\right) =-\left((u_j^{n+1}- u_j^n)\times B, \nabla \psi_h\right), \quad  \forall \psi_h\in S_h,
\end{align}
\noindent where 
$\text{Intp}(u_{j}^{n+1};v_{h})= \frac{1}{N}\big(\frac{u_{j}^{n+1}-u_{j}^{n}}{\Delta t}-u_{j,t}%
(t^{n+1}),v_{h}\big)$. 

Let
\begin{align}
&e_{u,j}^{n}=u_{j}^{n}-u_{j,h}^{n}=(u_{j}^{n}-I_{h} u_{j}^{n})+(I_{h} u_{j}%
^{n}-u_{j,h}^{n})=\eta_{j}^{n}+U_{j,h}^{n} ,\\
&e_{\phi,j}^{n}=\phi_{j}^{n}-\phi_{j,h}^{n}=(\phi_{j}^{n}-I_{h} \phi_{j}^{n})+(I_{h} \phi_{j}
^{n}-\phi_{j,h}^{n})=\xi_{j}^{n}+\Phi_{j,h}^{n} ,
\end{align}

\noindent where $I_{h} u_{j}^{n} \in V_{h} $ is an interpolant of $u_{j}^{n}$
in $V_{h}$, and $I_{h} \phi_{j}^{n} \in S_{h} $ is an interpolant of $\phi_{j}^{n}$
in $S_{h}.$

Subtracting (\ref{eq: conv}) from (\ref{eq:convtrue1}) gives

\begin{align}
&\frac{1}{N}\Big(\frac{U_{j,h}^{n+1}-U_{j,h}^{n}}{ \Delta t},v_{h}\Big) +\frac{1}{M^2}(\nabla U
_{j,h}^{n+1},\nabla v_{h})+\frac{1}{N}
b^{*}(u_{j}^{n+1},u_{j}^{n+1},v_{h})\\
&-\frac{1}{N} b^{*}(\overline{u}_{h}^{n},u_{j,h}^{n+1},v_{h}) 
-\frac{1}{N} b^{*}(u_{j,h}^{n}-\overline{u}_{h}^n,u_{j,h}
^{n},v_{h})+(U_{j,h}^{n+1}\times B, v_h\times B)\nonumber\\
&-(\nabla \Phi_{j,h}^n, v_h\times B)-(p_{j}^{n+1},\nabla\cdot v_{h})\nonumber\\
&=-\frac{1}{N}(\frac{\eta_{j}^{n+1}-\eta_{j}^{n}}{\Delta t},v_{h}) -\frac{1}{M^2}(\nabla\eta
_{j}^{n+1},\nabla v_{h})-(\eta_j^{n+1}\times B, v_h\times B)\nonumber\\
&+(\nabla \xi_j^n, v_h\times B)
+\left(\nabla (\phi_j^{n+1}-\phi_j^n), v_h\times B\right)+\text{Intp}(u_j^{n+1};v_{h}),\nonumber
\end{align}

\noindent and subtracting \eqref{2} from \eqref{eq:convtrue2} yields
\begin{align}\label{qqqq}
&\left(-\nabla \Phi_j^{n+1}, \nabla \psi_h\right)+\left( U_j^n\times B, \nabla \psi_h\right)\\
 & \qquad=\left(\nabla \xi_j^{n+1}, \nabla \psi_h\right)-\left( \eta_j^n\times B, \nabla \psi_h\right)-\left((u_j^{n+1}- u_j^n)\times B, \nabla \psi_h\right) .\nonumber
\end{align}

\noindent Setting $v_{h}=U_{j,h}^{n+1}\in V_{h}$ and $\psi_{h}=\Phi_{j,h}^{n+1}\in S_{h}$, rearranging the nonlinear
terms and multiply \eqref{qqqq} by $2$, we have
\begin{equation}
\label{eq:err1}
\begin{split}
&\frac{1}{N\Delta t}\left(\frac{1}{2}||U_{j,h}^{n+1}||^{2}-\frac{1}{2}||U
_{j,h}^{n}||^{2}+\frac{1}{2}\|U_{j,h}^{n+1}-U_{j,h}^{n}\|^{2}\right)+\frac{1}{M^2}
||\nabla U_{j,h}^{n+1}||^{2}\\
&\qquad +\Vert U_{j,h}^{n+1}\times B\Vert^2-(\nabla \Phi_{j,h}^n, U_{j,h}^{n+1}\times B)=-\frac{1}{M^2}
(\nabla\eta_{j}^{n+1},\nabla U_{j,h}^{n+1})\\
&\quad-\frac{1}{N}
b^{*}(u_{j,h}^n-\overline{u}_h^n,u_{j,h}^{n+1}-u_{j,h}^{n},U_{j,h}^{n+1})-\frac{1}{N} b^{*}(u_{j}^{n+1},u_{j}^{n+1},U_{j,h}^{n+1})\\
&\quad+\frac{1}{N} b^{*}(u_{j,h}^{n}%
,u_{j,h}^{n+1},U_{j,h}^{n+1})+(p_{j}
^{n+1},\nabla\cdot U_{j,h}^{n+1})
-(\eta_j^{n+1}\times B, U_{j,h}^{n+1}\times B)\\
&\quad+(\nabla \xi_j^n, U_{j,h}^{n+1}\times B)
+\left(\nabla (\phi_j^{n+1}-\phi_j^n), U_{j,h}^{n+1}\times B\right)\\
&\quad-\frac{1}{N}(\frac{\eta_{j}^{n+1}-\eta_{j}^{n}}{ \Delta t}, U_{j,h}^{n+1}) 
+\text{Intp}(u_j^{n+1}; U_{j,h}^{n+1}),
\end{split}
\end{equation}

\noindent and 
\begin{align}\label{eq:err11111}
& 2\Vert \nabla \Phi_{j,h}^{n+1}\Vert^2 -2\left( U_j^n\times B, \nabla \Phi_{j,h}^{n+1}\right)=-2\left(\nabla \xi_j^{n+1}, \nabla \Phi_{j,h}^{n+1}\right)\\
 & \qquad+2\left( \eta_j^n\times B, \nabla \Phi_{j,h}^{n+1}\right)+2\left((u_j^{n+1}- u_j^n)\times B, \nabla \Phi_{j,h}^{n+1}\right) .\nonumber
\end{align}

\noindent Adding \eqref{eq:err1} and \eqref{eq:err11111} and using equality \eqref{qqq2} gives

\begin{equation}
\label{eq:err111}
\begin{split}
&\frac{1}{N\Delta t}\left(\frac{1}{2}||U_{j,h}^{n+1}||^{2}-\frac{1}{2}||U
_{j,h}^{n}||^{2}+\frac{1}{2}\|U_{j,h}^{n+1}-U_{j,h}^{n}\|^{2}\right)+\frac{1}{M^2}
||\nabla U_{j,h}^{n+1}||^{2}\\
&\quad +\frac{1}{2}\left(\Vert U_{j,h}^{n+1}\times B\Vert^2-\Vert U_{j,h}^{n}\times B\Vert^2\right)+\frac{1}{2}\left(\Vert \nabla \Phi_{j,h}^{n+1}\Vert^2-\Vert \nabla \Phi_{j,h}^{n}\Vert^2 \right)\\
& \quad+\frac{1}{2}\left(\Vert -\nabla \Phi_{j,h}^n+U_{j,h}^{n+1}\times B \Vert^2+\Vert -\nabla \Phi_{j,h}^{n+1}+U_{j,h}^{n}\times B \Vert^2\right)+\Vert \nabla \Phi_{j,h}^{n+1}\Vert^2 \\
&=-\frac{1}{M^2}
(\nabla\eta_{j}^{n+1},\nabla U_{j,h}^{n+1})-\frac{1}{N} b^{*}(u_{j,h}^n-\overline{u}_h^n,u_{j,h}^{n+1}-u_{j,h}^{n},\xi_{j,h}^{n+1})\\
&\quad-\frac{1}{N} b^{*}(u_{j}^{n+1},u_{j}^{n+1},\xi_{j,h}^{n+1})+\frac{1}{N} b^{*}(u_{j,h}^{n}%
,u_{j,h}^{n+1},\xi_{j,h}^{n+1})+(p_{j}
^{n+1},\nabla\cdot U_{j,h}^{n+1})\\
&\quad -(-\nabla \xi_j^n+\eta_j^{n+1}\times B, U_{j,h}^{n+1}\times B)+(U_{j,h}^n\times B, \nabla \Phi_{j,h}^{n+1}
)\\
&\quad+\left(\nabla (\phi_j^{n+1}-\phi_j^n), U_{j,h}^{n+1}\times B\right)-2\left(\nabla \xi_j^{n+1}, \nabla \Phi_{j,h}^{n+1}\right)\\
&\quad +2\left( \eta_j^n\times B, \nabla \Phi_{j,h}^{n+1}\right)+2\left((u_j^{n+1}- u_j^n)\times B, \nabla \Phi_{j,h}^{n+1}\right)\\
&\quad-(\frac{\eta_{j}^{n+1}-\eta_{j}^{n}}{ \Delta t}, U_{j,h}^{n+1}) 
+\text{Intp}(u_j^{n+1}; U_{j,h}^{n+1}).
\end{split}
\end{equation}

\noindent We bound the terms on the right hand side of \eqref{eq:err1} as follows.

\begin{equation}
\begin{aligned}
-\frac{1}{M^2}(\nabla\eta_{j}^{n+1},\nabla U_{j,h}^{n+1}) &\leq \frac{1}{M^2}\|\nabla\eta_{j}
^{n+1}\| \|\nabla U_{j,h}^{n+1}\| \\
&\leq \frac{1}{4C_0M^2}\|\nabla\eta_{j}^{n+1}\|^{2}+ \frac{C_0}{M^2}\|\nabla U_{j,h}
^{n+1}\|^{2} .
\end{aligned}
\end{equation}

\noindent Next we analyze the nonlinear terms in \eqref{eq:err1} one by one. 
For the first nonlinear term, we have 
\begin{equation}\label{ineq:err3}
\begin{aligned}
&- \frac{1}{N} b^*(u_{j,h}^n-\overline{u}_h^n, u_{j,h}^{n+1}-u_{j,h}^n, U_{j,h}^{n+1})\\
=&\frac{1}{N} b^*(u_{j,h}^n-\overline{u}_h^n, e_{u,j}^{n+1}-e_{u,j}^n, U_{j,h}^{n+1}) -\frac{1}{N} b^*(u_{j,h}^n-\overline{u}_h^n, u_{j}^{n+1}-u_{j}^n, U_{j,h}^{n+1})\\
=&\frac{1}{N} b^*(u_{j,h}^n-\overline{u}_h^n, \eta_{j}^{n+1}, U_{j,h}^{n+1})-\frac{1}{N} b^*(u_{j,h}^n-\overline{u}_h^n, \eta_{j}^{n}, U_{j,h}^{n+1})\\
&\quad-\frac{1}{N} b^*(u_{j,h}^n-\overline{u}_h^n, U_{j,h}^{n}, U_{j,h}^{n+1})-\frac{1}{N} b^*(u_{j,h}^n-\overline{u}_h^n, u_{j}^{n+1}-u_{j}^n, U_{j,h}^{n+1})\,.
\end{aligned}
\end{equation}

\noindent Using inequality \eqref{In1} and Young's inequality, we have the following estimates.

\begin{equation}
\begin{aligned}
\frac{1}{N} b^*(u_{j,h}^n-\overline{u}_h^n, &\eta_{j}^{n+1}, U_{j,h}^{n+1}) \leq \frac{C}{N}\|\nabla (u_{j,h}^n-\overline{u}_h^n)\|\|\nabla\eta_{j}^{n+1}\|\|\nabla U_{j,h}^{n+1}\|\\
&\leq
\frac{C^2M^2}{4C_0N^2} \|\nabla (u_{j,h}^n-\overline{u}_h^n)\|^{2}\|\nabla\eta_{j}^{n+1}\|^{2}
+\frac{C_0}{M^2}\|\nabla U_{j,h}^{n+1}\|^{2},
\end{aligned}
\end{equation}

\noindent and
\begin{equation}
\begin{aligned}
-\frac{1}{N} b^*(u_{j,h}^n-\overline{u}_h^n, \eta_{j}^{n}&, U_{j,h}^{n+1}) \leq \frac{C}{N}\|\nabla (u_{j,h}^n-\overline{u}_h^n)\|\|\nabla\eta_{j}^{n}\|\|\nabla U_{j,h}^{n+1}\|\\
&\leq
\frac{C^2M^2}{4C_0N^2} \|\nabla (u_{j,h}^n-\overline{u}_h^n)\|^{2}\|\nabla\eta_{j}^{n}\|^{2}
+\frac{ C_0}{M^2}\|\nabla U_{j,h}^{n+1}\|^{2}.
\end{aligned}
\end{equation}
Because $b^*(\cdot,\cdot,\cdot)$ is skew-symmetric, we have
\begin{align}
-b^{\ast}(u_{j,h}^n-\overline{u}_h^n,U_{j,h}^{n},U_{j,h}^{n+1})&=b^{\ast}(u_{j,h}^n-\overline{u}_h^n,U_{j,h}^{n+1},U_{j,h}^{n})\\
&=b^{\ast}(u_{j,h}^n-\overline{u}_h^n,U
_{j,h}^{n+1}-U_{j,h}^{n},U_{j,h}^{n}) \nonumber\\
&=-b^{\ast}(u_{j,h}^n-\overline{u}_h^n, U_{j,h}^{n},U_{j,h}^{n+1}-U_{j,h}^n)\,.\nonumber
\end{align}
Then, by inequality \eqref{In2}, we obtain
\begin{equation}\label{inver}
\begin{aligned}
 -\frac{1}{N}&b^{\ast}(u_{j,h}^n-\overline{u}_h^n,U_{j,h}^{n},U_{j,h}^{n+1})  \\
\leq &\frac{C}{N} \Vert \nabla (u_{j,h}^n-\overline{u}_h^n)\Vert\Vert \nabla U_{j,h}^{n}\Vert
\Vert \nabla (U_{j,h}^{n+1}-U_{j,h}^n)\Vert^{1/2}\Vert  U_{j,h}^{n+1}-U_{j,h}^n\Vert^{1/2}\\
\leq & \frac{C}{N}\Vert \nabla (u_{j,h}^n-\overline{u}_h^n)\Vert\Vert \nabla U_{j,h}^{n}\Vert (h)^{-1/2}\Vert  U_{j,h}^{n+1}-U_{j,h}^n\Vert\\
\leq & \frac{1}{4N\triangle t}\Vert U_{j,h}^{n+1}- U_{j,h}^{n}\Vert^{2}+\left(
C\frac{\triangle t}{Nh}\Vert\nabla (u_{j,h}^n-\overline{u}_h^n)\Vert^{2}\right)  \Vert\nabla
U_{j,h}^{n}\Vert^{2}.
\end{aligned}
\end{equation}
For the last nonlinear term in \eqref{ineq:err3}, we have

\begin{align}
-\frac{1}{N}b^{*}(u_{j,h}^n-&\overline{u}_h^n,u_{j}^{n+1}-u_{j}^{n}, U_{j,h}^{n+1})\\
& \leq \frac{C}{N}\|\nabla
(u_{j,h}^n-\overline{u}_h^n)\|\|\nabla(u_{j}^{n+1}-u_{j}^{n})\|\|\nabla U_{j,h}^{n+1}%
\|\nonumber\\
&\leq 
\frac{C^2M^2}{4C_0N^2}\|\nabla (u_{j,h}^n-\overline{u}_h^n)\|^{2}\|\nabla (u_{j}^{n+1}-u_{j}^{n})\|^{2} 
+\frac{C_0}{M^2}\|\nabla U_{j,h}^{n+1}\|^{2} \nonumber\\
&\leq
\frac{C^2 M^2\Delta t}{4C_0N^2}\|\nabla (u_{j,h}^n-\overline{u}_h^n)\|^{2}\left(\int_{t^{n}}^{t^{n+1}}\| \nabla u_{j,t}\|^{2} \, dt\right) 
+\frac{C_0}{M^2} \|\nabla U_{j,h}^{n+1}\|^{2}
.\nonumber
\end{align}
Next, we bound the last two nonlinear terms on the RHS of \eqref{eq:err1} as follows:
\begin{equation}\label{eq:nonlinear}
\begin{aligned}
-&b^{*}(u_{j}^{n+1}, u_{j}^{n+1},U_{j,h}^{n+1})
+b^{*}(u_{j,h}^{n},u_{j,h}^{n+1},U_{j,h}^{n+1})\\
& = -b^{*}(e_{u,j}^{n},u_{j}^{n+1},U_{j,h}^{n+1})
-b^{*}(u_{j,h}^{n},e_{u,j}^{n+1},U_{j,h}^{n+1})
-b^{*}(u^{n+1}_{j}-u_j^{n}, u_{j}^{n+1}, U_{j,h}^{n+1}) \\
&= -b^{*}(\eta^{n}_j,u_{j}^{n+1},U_{j,h}^{n+1})-b^{*}(U_{j,h}^{n},u_{j}^{n+1},U_{j,h}^{n+1})\\
&\qquad -b^{*}(u^{n}_{j,h}, \eta_{j}^{n+1}, U_{j,h}^{n+1})-b^{*}(u^{n+1}_{j}-u^{n}_j,
u_{j}^{n+1}, U_{j,h}^{n+1}) .\nonumber
\end{aligned}
\end{equation}

\noindent With the assumption $u_j^{n+1}\in L^{\infty}(0,T; H^1(\Omega))$, we have 
\begin{equation}
\begin{aligned}
-\frac{1}{N} b^{*}(\eta_{j}^{n},u_{j}^{n+1},U_{j,h}^{n+1}) 
&\leq 
\frac{C}{N}\|\nabla \eta_{j}^{n}\|\|\nabla u_{j}^{n+1}\|\|\nabla U_{j,h}^{n+1}\|  \\
&\leq
\frac{C^2M^2}{4C_0N^2}\|\nabla\eta_{j}^{n}\|^{2} + \frac{C_0}{M^2}\|\nabla U_{j,h}^{n+1}\|^{2}.
\end{aligned}
\end{equation}
Using the inequality \eqref{In3}, Young's inequality, and $u_j^{n+1}\in L^{\infty}(0,T; H^1(\Omega))$, we get
\begin{align}
-\frac{1}{N} b^{*}(U^{n}_{j,h}, u_{j}^{n+1}, U_{j,h}^{n+1}) &\leq \frac{C}{N}\| \nabla U
^{n}_{j,h} \|^{1/2} \Vert U^{n}_{j,h}\Vert^{1/2}\|\nabla
u_{j}^{n+1}\|\|\nabla U_{j,h}^{n+1}\|\nonumber\\
&\leq \frac{C}{N}\| \nabla U^{n}_{j,h} \|^{1/2} \Vert U^{n}_{j,h}\Vert
^{1/2}\|\nabla U_{j,h}^{n+1}\|\nonumber\\
&\leq 
\frac{C}{N}\Big(\frac{1}{4\alpha}\|\nabla U^{n}_{j,h} \|\| U^{n}_{j,h} \| + \alpha\|\nabla U_{j,h}^{n+1}\|^{2}\Big)\\
&\leq 
\frac{C}{N}\Big(\frac{1}{4\alpha} \big(\frac{\delta}{2}\|\nabla U^{n}_{j,h} \|^{2}+\frac{1}{2\delta}\| U^{n}_{j,h} \|^2 \big)
+ \alpha\|\nabla U_{j,h}^{n+1}\|^{2} \Big)\nonumber\\
&\leq
\frac{C_0}{M^2}\|\nabla  U^{n}_{j,h} \|^{2}+\frac{C^4M^6}{64C_0^3N^4 }\| U^{n}_{j,h} \|^{2} + \frac{C_0}{M^2} \|\nabla U_{j,h}^{n+1}\|^{2}\, ,\nonumber
\end{align}
where we set $\alpha= \frac{C_0N}{CM^2}$ and $\delta= \frac{8C_0^2N^2}{C^2M^4}$. 
By Young's inequality, inequality \eqref{In3}, and the result \eqref{qq7} from the stability analysis, i.e., $\Vert u_{j,h}^n \Vert^2 \leq C$, we also have
\begin{equation}
\begin{aligned}
\frac{1}{N} b^{*}(u^{n}_{j,h}, \eta_{j}^{n+1}, U_{j,h}^{n+1}) 
&\leq 
\frac{C}{N}\|\nabla u^{n}_{j,h}\|^{1/2}\Vert u_{j,h}^n\Vert^{1/2}\|\nabla \eta_{j}^{n+1}\|\|\nabla U_{j,h}^{n+1}\|\\
&\leq
\frac{C^2M^2}{4C_0N^2}\|\nabla u_{j,h}^{n}\|\|\nabla\eta^{n+1}_{j}\|^{2} + \frac{C_0}{M^2}\|\nabla U_{j,h}^{n+1}\|^{2},
\end{aligned}
\end{equation}
and 
\begin{align}
\frac{1}{N}b^{*}(u_{j}^{n+1}-u_{j}^{n},u_{j}^{n+1},&U_{j,h}^{n+1}) 
\leq 
 \frac{C}{N}\|\nabla (u_{j}^{n+1}-u_{j}^{n})\|\|\nabla u_{j}^{n+1}\|\|\nabla U_{j,h}
^{n+1}\|\nonumber\\
\leq & \frac{C^2M^2}{4C_0N^2}\|\nabla(u_{j}^{n+1}-u_{j}^{n})\|^{2} + \frac{C_0}{M^2} \|\nabla U_{j,h}^{n+1}\|^{2} \nonumber\\
= & \frac{C^2M^2\Delta t^2}{4C_0N^2}\left\|\frac{\nabla u_{j}^{n+1}- \nabla u_{j}^{n}}{\Delta t}\right\|^{2} + \frac{C_0}{M^2}\|\nabla U_{j,h}^{n+1}\|^{2} \\
= &\frac{C^2M^2\Delta t^2}{4C_0N^2}
\int_{\Omega}\left(\frac{1}{\Delta t}\int_{t^{n}}^{t^{n+1}} \nabla u_{j,t}
\, dt\right)^{2} d\Omega+ \frac{C_0}{M^2} \|\nabla U_{j,h}^{n+1}\|^{2} \nonumber\\
\leq & \frac{C^2 M^2\Delta t}{4C_0N^2} \int_{t^{n}}^{t^{n+1}}\| \nabla u_{j,t} \|^{2}\, dt
 + \frac{C_0}{M^2}  \|\nabla U_{j,h}^{n+1}\|^{2}. \nonumber
\end{align}
For the pressure term in \eqref{eq:err111}, because $U_{j,h}^{n+1}\in V_{h}$, for $\forall q_{j,h}^{n+1}\in Q_h $ we have

\begin{align}
(p_{j}^{n+1},\nabla\cdot U_{j,h}^{n+1})&=(p_{j}^{n+1}-q_{j,h}^{n+1},
\nabla\cdot U_{j,h}^{n+1})\nonumber\\
&\leq\sqrt{d}\,\|p_{j}^{n+1}-q_{j,h}^{n+1}\|\|\nabla U_{j,h}^{n+1}\|\\
&\leq
\frac{M^2d}{4\, C_0} \|p_{j}^{n+1}-q_{j,h}^{n+1}\|^{2} + \frac{C_0}{M^2}\|\nabla U_{j,h}^{n+1}\|^{2}
 \text{ .}\nonumber
\end{align}
The other terms are bounded as follows.
\begin{align}
-\frac{1}{N}\Big( \frac{\eta_{j}^{n+1}-\eta_{j}^{n}}{ \Delta t},U_{j,h}^{n+1} \Big) 
&\leq
\frac{C}{N} \Big\|\frac{\eta_{j}^{n+1}-\eta_{j}^{n}}{ \Delta t} \Big\| \|\nabla U_{j,h}
^{n+1}\|\nonumber\\
&\leq 
\frac{C^2M^2}{4C_0N^2} \left \|\frac{\eta_{j}^{n+1}-\eta_{j}^{n}}{ \Delta t} \right\|^{2}
+\frac{C_0}{M^2} \|\nabla U_{j,h}^{n+1}\|^{2}\\
&\leq 
\frac{C^2M^2}{4C_0N^2} \left \|\frac{1}{\Delta t}\int_{t^{n}}^{t^{n+1}} \eta_{j,t} \text{ }
dt \right \|^2+\frac{C_0}{M^2}\|\nabla U_{j,h}^{n+1}\|^{2}\nonumber\\
&\leq
\frac{C^2M^2}{4C_0N^2\Delta t}\int_{t^{n}}^{t^{n+1}}\| \eta_{j,t}\|^{2}\text{ }
dt+\frac{C_0}{M^2} \|\nabla U_{j,h}^{n+1}\|^{2}.\nonumber
\end{align}

\begin{align}
\text{Intp}(u_{j}^{n+1}; U_{j,h}^{n+1})
&=\frac{1}{N} \left(\frac{u_{j}^{n+1}-u_{j}^{n}}{\Delta
t}-u_{j,t}(t^{n+1}), U_{j,h}^{n+1}\right)\nonumber\\
&\leq \frac{C}{N}\left\|\frac{u_{j}^{n+1}-u_{j}^{n}}{\Delta t}-u_{j,t}(t^{n+1})\right\| \|\nabla
U_{j,h}^{n+1}\|\nonumber\\
&\leq
\frac{C^2M^2}{4C_0N^2}\left \|\frac{u_{j}^{n+1}-u_{j}^{n}}{\Delta t}-u_{j,t}(t^{n+1}) \right\|^{2}
+ \frac{C_0}{M^2}\|\nabla U_{j,h}^{n+1}\|^{2} \nonumber \\
&\leq
\frac{C^2M^2\Delta t}{4C_0N^2}\int_{t^{n}}^{t^{n+1}}\|u_{j,tt}\|^{2}\, dt 
+ \frac{C_0}{M^2}\|\nabla U_{j,h}^{n+1}\|^{2}
.
\label{lastineq}
\end{align}

\begin{align}
-(-\nabla \xi_j^n+\eta_j^{n+1}\times B,& U_{j,h}^{n+1}\times B)\leq \Vert -\nabla \xi_j^n+\eta_j^{n+1}\times B\Vert \Vert U_{j,h}^{n+1}\times B\Vert\\
&\leq \frac{1}{4}\Vert -\nabla \xi_j^n+\eta_j^{n+1}\times B\Vert^2+\Vert B\Vert_{L^{\infty}}^2\Vert U_{j,h}^{n+1}\Vert^2\nonumber\\
&\leq \frac{1}{2}\Vert \nabla \xi_j^n \Vert^2+\frac{1}{2}\Vert B\Vert_{L^{\infty}}^2\Vert \eta_j^{n+1}\Vert^2+\Vert B\Vert_{L^{\infty}}^2\Vert U_{j,h}^{n+1}\Vert^2.\nonumber
\end{align}

\begin{align}
(U_{j,h}^n\times B, \nabla \Phi_{j,h}^{n+1})&\leq \Vert U_{j,h}^n\times B\Vert \Vert \nabla \Phi_{j,h}^{n+1}\Vert \\
&\leq  \frac{1}{4C_1}\Vert B\Vert_{L^{\infty}}^2\Vert U_{j,h}^{n}\Vert^2 + C_1\Vert \nabla \Phi_{j,h}^{n+1}\Vert^2.\nonumber
\end{align}

\begin{align}
\left(\nabla (\phi_j^{n+1}-\phi_j^n), U_{j,h}^{n+1}\times B\right)&\leq \Vert  \nabla (\phi_j^{n+1}-\phi_j^n)\Vert\Vert  U_{j,h}^{n+1}\times B\Vert\\
&\leq \frac{1}{4}\Vert  \nabla (\phi_j^{n+1}-\phi_j^n)\Vert^2+\Vert B\Vert_{L^{\infty}}^2\Vert U_{j,h}^{n+1}\Vert^2\nonumber\\
&\leq \frac{1}{4}\Vert   \int_{t^n}^{t^{n+1}} \nabla \phi_{j,t} \, dt \Vert^2+\Vert B\Vert_{L^{\infty}}^2\Vert U_{j,h}^{n+1}\Vert^2\nonumber\\
&\leq \frac{1}{4}\Delta t \int_{t^n}^{t^{n+1}}\Vert \nabla \phi_{j,t}\Vert^2 dt+\Vert B\Vert_{L^{\infty}}^2\Vert U_{j,h}^{n+1}\Vert^2.\nonumber
\end{align}

\begin{align}
-2\left(\nabla \xi_j^{n+1}, \nabla \Phi_{j,h}^{n+1}\right)&\leq 2 \Vert \nabla \xi_j^{n+1}\Vert\Vert \nabla \Phi_{j,h}^{n+1}\Vert\\
&\leq \frac{1}{C_1}\Vert \nabla \xi_j^{n+1}\Vert^2+C_1 \Vert \nabla \Phi_{j,h}^{n+1}\Vert^2.\nonumber
\end{align}

\begin{align}
 2\left( \eta_j^n\times B, \nabla \Phi_{j,h}^{n+1}\right)&\leq 2 \Vert \eta_j^n\times B\Vert \Vert \nabla \Phi_{j,h}^{n+1} \Vert\\
 &\leq \frac{1}{C_1} \Vert \eta_j^n\times B\Vert^2+C_1 \Vert  \nabla \Phi_{j,h}^{n+1} \Vert^2\nonumber\\
 &\leq \frac{1}{C_1} \Vert  B\Vert_{L^{\infty}}^2\Vert \eta_j^n\Vert^2+C_1 \Vert  \nabla \Phi_{j,h}^{n+1} \Vert^2.\nonumber
\end{align}

\begin{align}\label{lasteq}
2\left((u_j^{n+1}- u_j^n)\times B, \nabla \Phi_{j,h}^{n+1}\right)&\leq 2 \Vert (u_j^{n+1}- u_j^n)\times B\Vert\Vert \nabla \Phi_{j,h}^{n+1}\Vert\\
&\leq \frac{1}{C_1}\Vert B\Vert^2_{L^{\infty}}\Vert u_j^{n+1}- u_j^n\Vert^2+C_1 \Vert \nabla \Phi_{j,h}^{n+1}\Vert^2\nonumber\\
&\leq \frac{\Delta t}{C_1}\Vert B\Vert^2_{L^{\infty}}\int_{t^n}^{t^{n+1}}\Vert u_{j,t}\Vert^2 dt +C_1 \Vert \nabla \Phi_{j,h}^{n+1}\Vert^2.\nonumber
\end{align}

\noindent Combining \eqref{eq:err1}-\eqref{lasteq}, and taking $C_0=\frac{1}{24}, C_1=\frac{1}{8}$, we have

\begin{align}\label{eq:err2}
&\frac{1}{N\Delta t}\left(\frac{1}{2}||U_{j,h}^{n+1}||^{2}-\frac{1}{2}||U
_{j,h}^{n}||^{2}+\frac{1}{4}\|U_{j,h}^{n+1}-U_{j,h}^{n}\|^{2}\right)+\frac{1}{4M^2}
||\nabla U_{j,h}^{n+1}||^{2}\\
&\qquad +\frac{1}{2}\Vert U_{j,h}^{n+1}\times B\Vert^2-\frac{1}{2}\Vert U_{j,h}^{n}\times B\Vert^2+\frac{1}{2}\Vert \nabla \Phi_{j,h}^{n+1}\Vert^2-\frac{1}{2}\Vert \nabla \Phi_{j,h}^{n}\Vert^2 \nonumber\\
& +\frac{1}{2}\Vert -\nabla \Phi_{j,h}^n+U_{j,h}^{n+1}\times B \Vert^2+\frac{1}{2}\Vert -\nabla \Phi_{j,h}^{n+1}+U_{j,h}^{n}\times B\Vert^2+\frac{1}{2}\Vert \nabla \Phi_{j,h}^{n+1}\Vert^2 \nonumber\\
&+\frac{1}{4M^2}\left(
||\nabla U_{j,h}^{n+1}||^{2}-||\nabla U_{j,h}^{n}||^{2}\right)+\frac{1}{24M^2}\left(
\Vert\nabla U_{j,h}^{n+1}\Vert^{2}-\Vert\nabla U_{j,h}^{n}\Vert^{2}\right)\nonumber\\
&+\frac{1}{4M^2}\left(1- C\frac{M^2}{N}\frac{\triangle t}{h}\Vert\nabla (u_{j,h}^n-\overline{u}_h^n)\Vert^{2}\right)
||\nabla U_{j,h}^{n}||^{2}\nonumber\\
&\leq \frac{1}{4C_0M^2}\|\nabla\eta_{j}^{n+1}\|^{2}+\frac{C^2M^2}{4C_0N^2} \|\nabla (u_{j,h}^n-\overline{u}_h^n)\|^{2}\|\nabla\eta_{j}^{n+1}\|^{2}\nonumber\\
&+\frac{C^2M^2}{4C_0N^2} \|\nabla (u_{j,h}^n-\overline{u}_h^n)\|^{2}\|\nabla\eta_{j}^{n}\|^{2}+\frac{C^2 M^2\Delta t}{4C_0N^2}\|\nabla (u_{j,h}^n-\overline{u}_h^n)\|^{2}\left(\int_{t^{n}}^{t^{n+1}}\| \nabla u_{j,t}\|^{2} \, dt\right) \nonumber\\
&+\frac{C^2M^2}{4C_0N^2}\|\nabla\eta_{j}^{n}\|^{2}+\frac{C_0}{M^2}\|\nabla  U^{n}_{j,h} \|^{2}+\frac{C^4M^6}{64C_0^3N^4 }\| U^{n}_{j,h} \|^{2} +\frac{C^2M^2}{4C_0N^2}\|\nabla u_{j,h}^{n}\|\|\nabla\eta^{n+1}_{j}\|^{2}\nonumber\\
&+\frac{C^2 M^2\Delta t}{4C_0N^2} \int_{t^{n}}^{t^{n+1}}\| \nabla u_{j,t} \|^{2}\, dt+\frac{M^2d}{4\, C_0} \|p_{j}^{n+1}-q_{j,h}^{n+1}\|^{2}+\frac{C^2M^2}{4C_0N^2\Delta t}\int_{t^{n}}^{t^{n+1}}\| \eta_{j,t}\|^{2}\text{ }
dt\nonumber\\
&+\frac{C^2M^2\Delta t}{4C_0N^2}\int_{t^{n}}^{t^{n+1}}\|u_{j,tt}\|^{2}\, dt+\frac{1}{2}\Vert \nabla \xi_j^n \Vert^2+\frac{1}{2}\Vert B\Vert_{L^{\infty}}^2\Vert \eta_j^{n+1}\Vert^2+2\Vert B\Vert_{L^{\infty}}^2\Vert U_{j,h}^{n+1}\Vert^2\nonumber\\
& + \frac{1}{4C_1}\Vert B\Vert_{L^{\infty}}^2\Vert U_{j,h}^{n}\Vert^2 +\frac{1}{4}\Delta t \int_{t^n}^{t^{n+1}}\Vert \nabla \phi_{j,t}\Vert^2 dt+\frac{1}{C_1}\Vert \xi_j^{n+1}\Vert^2\nonumber\\
&+\frac{1}{C_1} \Vert  B\Vert_{L^{\infty}}^2\Vert \eta_j^n\Vert^2+\frac{\Delta t}{C_1}\Vert B\Vert^2_{L^{\infty}}\int_{t^n}^{t^{n+1}}\Vert u_{j,t}\Vert^2 dt .\nonumber
\end{align}

\noindent By the convergence condition \eqref{conv1}, we have 

\begin{align*}
1- C\frac{M^2}{N}\frac{\triangle t}{h}\Vert\nabla (u_{j,h}^n-\overline{u}_h^n)\Vert^{2}\geq 0.
\end{align*}

\noindent Then, after rearranging terms, \eqref{eq:err2} reduces to
\begin{align}\label{eq:err22222}
&\frac{1}{N\Delta t}\left(\frac{1}{2}||U_{j,h}^{n+1}||^{2}-\frac{1}{2}||U
_{j,h}^{n}||^{2}+\frac{1}{4}\|U_{j,h}^{n+1}-U_{j,h}^{n}\|^{2}\right)+\frac{1}{4M^2}
||\nabla U_{j,h}^{n+1}||^{2}\\
&\qquad +\frac{1}{2}\Vert U_{j,h}^{n+1}\times B\Vert^2-\frac{1}{2}\Vert U_{j,h}^{n}\times B\Vert^2+\frac{1}{2}\Vert \nabla \Phi_{j,h}^{n+1}\Vert^2-\frac{1}{2}\Vert \nabla \Phi_{j,h}^{n}\Vert^2 \nonumber\\
& +\frac{1}{2}\Vert -\nabla \Phi_{j,h}^n+U_{j,h}^{n+1}\times B \Vert^2+\frac{1}{2}\Vert -\nabla \Phi_{j,h}^{n+1}+U_{j,h}^{n}\times B\Vert^2 \nonumber\\
&+\frac{1}{2}\Vert \nabla \Phi_{j,h}^{n+1}\Vert^2+\frac{1}{4M^2}\left(
\Vert\nabla U_{j,h}^{n+1}\Vert^{2}-\Vert\nabla U_{j,h}^{n}\Vert^{2}\right)+\frac{1}{24M^2}\left(
\Vert\nabla U_{j,h}^{n+1}\Vert^{2}-\Vert\nabla U_{j,h}^{n}\Vert^{2}\right)\nonumber\\
&\leq \left(\frac{CM^6}{N^4 }+6 \Vert B\Vert_{L^{\infty}}^2\right)\| U^{n}_{j,h} \|^{2}+2\Vert B\Vert_{L^{\infty}}^2\Vert U_{j,h}^{n+1}\Vert^2+\frac{6}{M^2}\|\nabla\eta_{j}^{n+1}\|^{2}\nonumber\\
&+\frac{Ch}{N\Delta t} \|\nabla\eta_{j}^{n+1}\|^{2}+\frac{Ch}{N\Delta t} \|\nabla\eta_{j}^{n}\|^{2}+\frac{Ch}{N}\left(\int_{t^{n}}^{t^{n+1}}\| \nabla u_{j,t}\|^{2} \, dt\right) +\frac{CM^2}{N^2}\|\nabla\eta_{j}^{n}\|^{2}\nonumber\\
& +\frac{CM^2}{N^2}\|\nabla u_{j,h}^{n}\|\|\nabla\eta^{n+1}_{j}\|^{2}+\frac{C M^2\Delta t}{N^2} \int_{t^{n}}^{t^{n+1}}\| \nabla u_{j,t} \|^{2}\, dt+18M^2 \|p_{j}^{n+1}-q_{j,h}^{n+1}\|^{2}\nonumber\\
&+\frac{CM^2}{N^2\Delta t}\int_{t^{n}}^{t^{n+1}}\| \eta_{j,t}\|^{2}\text{ }
dt+\frac{CM^2\Delta t}{N^2}\int_{t^{n}}^{t^{n+1}}\|u_{j,tt}\|^{2}\, dt+\frac{1}{2}\Vert \nabla \xi_j^n \Vert^2\nonumber\\
&  +\frac{1}{2}\Vert B\Vert_{L^{\infty}}^2\Vert \eta_j^{n+1}\Vert^2+\frac{1}{4}\Delta t \int_{t^n}^{t^{n+1}}\Vert \nabla \phi_{j,t}\Vert^2 dt+24\Vert \xi_j^{n+1}\Vert^2\nonumber\\
&+ 24\Vert  B\Vert_{L^{\infty}}^2\Vert \eta_j^n\Vert^2+24\Delta t\Vert B\Vert^2_{L^{\infty}}\int_{t^n}^{t^{n+1}}\Vert u_{j,t}\Vert^2 dt .\nonumber
\end{align}

\noindent Summing (\ref{eq:err22222}) and multiplying both sides by $2N\Delta t$ gives
\begin{align}\label{eq:err3}
&\Vert U_{j,h}^{n}\Vert^{2}+\frac{1}{2}\sum_{l=0}^{n-1}\|U_{j,h}^{l+1}-U_{j,h}^{l}\|^{2}+\Delta t\sum_{l=0}^{n-1}
\frac{N}{2M^2}\Vert\nabla U_{j,h}^{l+1}\Vert^{2}\\
&\qquad +N\Delta t\Vert U_{j,h}^{n}\times B\Vert^2+N\Delta t\Vert \nabla \Phi_{j,h}^{n}\Vert^2 
+N\Delta t\sum_{l=0}^{n-1}\Vert -\nabla \Phi_{j,h}^l+U_{j,h}^{l+1}\times B \Vert^2\nonumber\\
&\qquad+N\Delta t\sum_{l=0}^{n-1}\Vert -\nabla \Phi_{j,h}^{l+1}+U_{j,h}^{l}\times B \Vert^2+N\Delta t\sum_{l=0}^{n-1}\Vert \nabla \Phi_{j,h}^{l+1}\Vert^2 
+\frac{7N}{12M^2}\Delta t
\Vert\nabla U_{j,h}^{n}\Vert^{2}\nonumber\\
&\leq \Delta t\sum_{l=0}^{n-1}\left(\frac{CM^6}{N^3 }+16N \Vert B\Vert_{L^{\infty}}^2\right)\|\nabla  U^{l+1}_{j,h} \|^{2}+\Delta t\left(\frac{CM^6}{N^3 }+12N \Vert B\Vert_{L^{\infty}}^2\right)\| U^{0}_{j,h} \|^{2}\nonumber\\
&+  \Vert U
_{j,h}^{0}\Vert^{2}+N\Delta t\Vert U_{j,h}^{0}\times B\Vert^2+N\Delta t\Vert \nabla \Phi_{j,h}^{0}\Vert^2+\frac{7N}{12M^2}\Delta t
\Vert\nabla U_{j,h}^{0}\Vert^{2}\nonumber\\
&+\Delta t\sum_{l=0}^{n-1}
\bigg\{\frac{12N}{M^2}\|\nabla\eta_{j}^{l+1}\|^{2}+\frac{Ch}{\Delta t} \|\nabla\eta_{j}^{l+1}\|^{2}+\frac{Ch}{\Delta t} \|\nabla\eta_{j}^{l}\|^{2}+Ch\left(\int_{t^{l}}^{t^{l+1}}\| \nabla u_{j,t}\|^{2} \, dt\right) \nonumber\\
&+\frac{CM^2}{N}\|\nabla\eta_{j}^{l}\|^{2} +\frac{CM^2}{N}\|\nabla u_{j,h}^{l}\|\|\nabla\eta^{l+1}_{j}\|^{2}+\frac{C M^2\Delta t}{N} \int_{t^{l}}^{t^{l+1}}\| \nabla u_{j,t} \|^{2}\, dt\nonumber\\
&+36M^2N \|p_{j}^{l+1}-q_{j,h}^{l+1}\|^{2}+\frac{CM^2}{N\Delta t}\int_{t^{l}}^{t^{l+1}}\| \eta_{j,t}\|^{2}\text{ }
dt+\frac{CM^2\Delta t}{N}\int_{t^{l}}^{t^{l+1}}\|u_{j,tt}\|^{2}\, dt\nonumber\\
&+N\Vert \nabla \xi_j^l \Vert^2  +N\Vert B\Vert_{L^{\infty}}^2\Vert \eta_j^{l+1}\Vert^2+\frac{N}{2}\Delta t \int_{t^l}^{t^{l+1}}\Vert \nabla \phi_{j,t}\Vert^2 dt+48N\Vert \xi_j^{l+1}\Vert^2\nonumber\\
&+ 48N\Vert  B\Vert_{L^{\infty}}^2\Vert \eta_j^l\Vert^2+48N\Delta t\Vert B\Vert^2_{L^{\infty}}\int_{t^l}^{t^{l+1}}\Vert u_{j,t}\Vert^2 dt \bigg\}.\nonumber
\end{align}

Using the interpolation inequality \eqref{interp2} and  the result \eqref{qq7} from the stability analysis, i.e., $\Delta t \sum_{l=0}^{n-1}\Vert \nabla u_{j,h}^{l+1} \Vert^2 \leq CM$, we have
\begin{align}
\frac{CM^2}{N}\Delta t\sum_{l=0}^{n-1}\Vert\nabla\eta_{j}^{l+1}\Vert^{2}&\Vert\nabla u_{j,h}^{l}\Vert 
\leq C\frac{M^2}{N} h^{2k}\Delta t \sum_{l=0}^{n-1}\Vert u_j^{l+1}\Vert^2_{k+1}\Vert^{2}\Vert\nabla u_{j,h}^{l}\Vert\label{ineq:err4}\\
\leq & C\frac{M^2}{N} h^{2k}\left( \Delta t \sum_{l=0}^{n-1}\Vert u_{j}^{l+1}\Vert_{k+1}^4+\Delta t \sum_{l=0}^{n-1}\Vert\nabla u_{j,h}^{l+1}\Vert^2 \right)\nonumber\\
\leq & C\frac{M^2}{N} h^{2k}\vertiii{ u_j}^4_{4, k+1}+C\frac{M^3}{N} h^{2k}.\nonumber
\end{align}
Applying the interpolation inequalities \eqref{Interp1}, \eqref{interp2}, and \eqref{interp3} gives
\begin{align}\label{ineq:errlast}
&\Vert U_{j,h}^{n}\Vert^{2}+\frac{1}{2}\sum_{l=0}^{n-1}\|U_{j,h}^{l+1}-U_{j,h}^{l}\|^{2}+\Delta t\sum_{l=0}^{n-1}
\frac{N}{2M^2}\Vert\nabla U_{j,h}^{l+1}\Vert^{2}\\
&\qquad +N\Delta t\Vert U_{j,h}^{n}\times B\Vert^2+N\Delta t\Vert \nabla \Phi_{j,h}^{n}\Vert^2 
+N\Delta t\sum_{l=0}^{n-1}\Vert -\nabla \Phi_{j,h}^l+U_{j,h}^{l+1}\times B \Vert^2\nonumber\\
&\qquad+N\Delta t\sum_{l=0}^{n-1}\Vert -\nabla \Phi_{j,h}^{l+1}+U_{j,h}^{l}\times B \Vert^2+N\Delta t\sum_{l=0}^{n-1}\Vert \nabla \Phi_{j,h}^{l+1}\Vert^2 
+\frac{7N}{12M^2}\Delta t
\Vert\nabla U_{j,h}^{n}\Vert^{2}\nonumber\\
&\leq \Delta t\sum_{l=0}^{n-1}\left(\frac{CM^6}{N^3 }+16N \Vert B\Vert_{L^{\infty}}^2\right)\|\nabla  U^{l+1}_{j,h} \|^{2}+\Delta t\left(\frac{CM^6}{N^3 }+12N \Vert B\Vert_{L^{\infty}}^2\right)\| U^{0}_{j,h} \|^{2}\nonumber\\
&+  \Vert U
_{j,h}^{0}\Vert^{2}+N\Delta t\Vert U_{j,h}^{0}\times B\Vert^2+N\Delta t\Vert \nabla \Phi_{j,h}^{0}\Vert^2+\frac{7N}{12M^2}\Delta t
\Vert\nabla U_{j,h}^{0}\Vert^{2}\nonumber\\
&+C\frac{N}{M^2} h^{2k}\vertiii{ u_j }^2_{2,k+1}+C h^{2k+1}\Delta t^{-1}\vertiii{ u_j }^2_{2,k+1}+ Ch\Delta t\vertiii{ \nabla u_{j,t}}^2_{2,0}
\nonumber\\
&+C\frac{M^2}{N} h^{2k}\vertiii{ u_j }^2_{2,k+1}+C\frac{M^2}{N} h^{2k}\vertiii{ u_j}^4_{4, k+1}+C\frac{M^3}{N} h^{2k}+C\frac{M^2}{N}\Delta t^2\vertiii{ \nabla u_{j,t}}^2_{2,0}\nonumber\\
&+CM^2N h^{2s+2}\vertiii{p_{j}}_{2,s+1}^{2}+C\frac{M^2}{N}h^{2k+2}\vertiii{u_{j,t}}_{2,k+1}^{2} +C\frac{M^2}{N}\Delta t^{2}\vertiii{u_{j,tt}}_{2,0}^{2}\nonumber\\
& +C N h^{2k}\vertiii{ \phi_j }^2_{2,k+1}+CN\Vert B\Vert_{L^{\infty}}^2 h^{2k+2}\vertiii{ u_j }^2_{2,k+1}+ CN\Delta t^2\vertiii{ \nabla \phi_{j,t}}^2_{2,0}\nonumber\\
&+C Nh^{2k+2}\vertiii{ \phi_j }^2_{2,k+1}+CN\Vert B\Vert_{L^{\infty}}^2\Delta t^2\vertiii{ u_{j,t}}^2_{2,0}.\nonumber
\end{align}

\noindent Let $\Delta t$ be sufficiently small, i.e., $\Delta t
<\left(\frac{CM^6}{N^3 }+16N \Vert B\Vert_{L^{\infty}}^2\right)^{-1}$. We can apply the lemma \eqref{lm:Gronwall}, denoting $\tilde{C}=\frac{CM^6}{N^3 }+16N \Vert B\Vert_{L^{\infty}}^2$, and obtain

\begin{align}\label{ineq:errlast1}
&\Vert U_{j,h}^{n}\Vert^{2}+\frac{1}{2}\sum_{l=0}^{n-1}\|U_{j,h}^{l+1}-U_{j,h}^{l}\|^{2}+\Delta t\sum_{l=0}^{n-1}
\frac{N}{2M^2}\Vert\nabla U_{j,h}^{l+1}\Vert^{2}\\
&\qquad +N\Delta t\Vert U_{j,h}^{n}\times B\Vert^2+N\Delta t\Vert \nabla \Phi_{j,h}^{n}\Vert^2 
+N\Delta t\sum_{l=0}^{n-1}\Vert -\nabla \Phi_{j,h}^l+U_{j,h}^{l+1}\times B \Vert^2\nonumber\\
&\qquad+N\Delta t\sum_{l=0}^{n-1}\Vert -\nabla \Phi_{j,h}^{l+1}+U_{j,h}^{l}\times B \Vert^2+N\Delta t\sum_{l=0}^{n-1}\Vert \nabla \Phi_{j,h}^{l+1}\Vert^2 
+\frac{7N}{12M^2}\Delta t
\Vert\nabla U_{j,h}^{n}\Vert^{2}\nonumber\\
&\leq e^{\frac{T\tilde{C}}{1-\Delta t \tilde{C}}}
\bigg\{\Delta t\left(\frac{CM^6}{N^3 }+12N \Vert B\Vert_{L^{\infty}}^2\right)\| U^{0}_{j,h} \|^{2}\nonumber\\
&+  \Vert U
_{j,h}^{0}\Vert^{2}+N\Delta t\Vert U_{j,h}^{0}\times B\Vert^2+N\Delta t\Vert \nabla \Phi_{j,h}^{0}\Vert^2+\frac{7N}{12M^2}\Delta t
\Vert\nabla U_{j,h}^{0}\Vert^{2}\nonumber\\
&+C\frac{N}{M^2} h^{2k}\vertiii{ u_j }^2_{2,k+1}+C h^{2k+1}\Delta t^{-1}\vertiii{ u_j }^2_{2,k+1}+ Ch\Delta t\vertiii{ \nabla u_{j,t}}^2_{2,0}
\nonumber\\
&+C\frac{M^2}{N} h^{2k}\vertiii{ u_j }^2_{2,k+1}+C\frac{M^2}{N} h^{2k}\vertiii{ u_j}^4_{4, k+1}+C\frac{M^3}{N} h^{2k}+C\frac{M^2}{N}\Delta t^2\vertiii{ \nabla u_{j,t}}^2_{2,0}\nonumber\\
&+CM^2N h^{2s+2}\vertiii{p_{j}}_{2,s+1}^{2}+C\frac{M^2}{N}h^{2k+2}\vertiii{u_{j,t}}_{2,k+1}^{2} +C\frac{M^2}{N}\Delta t^{2}\vertiii{u_{j,tt}}_{2,0}^{2}\nonumber\\
& +C N h^{2k}\vertiii{ \phi_j }^2_{2,k+1}+CN\Vert B\Vert_{L^{\infty}}^2 h^{2k+2}\vertiii{ u_j }^2_{2,k+1}+ CN\Delta t^2\vertiii{ \nabla \phi_{j,t}}^2_{2,0}\nonumber\\
&+C Nh^{2k+2}\vertiii{ \phi_j }^2_{2,k+1}+CN\Vert B\Vert_{L^{\infty}}^2\Delta t^2\vertiii{ u_{j,t}}^2_{2,0}
\bigg\}.\nonumber
\end{align}

\noindent We now add the following terms to both sides of \eqref{ineq:errlast1}.

\begin{align}\label{extra}
&\Vert \eta_{j}^{n}\Vert^{2}+\frac{1}{2}\sum_{l=0}^{n-1}\|\eta_{j}^{l+1}-\eta_{j}^{l}\|^{2}+\Delta t\sum_{l=0}^{n-1}
\frac{N}{2M^2}\Vert\nabla \eta_{j}^{l+1}\Vert^{2}\\
&\qquad +N\Delta t\Vert \eta_{j}^{n}\times B\Vert^2+N\Delta t\Vert \nabla \xi_{j}^{n}\Vert^2 
+N\Delta t\sum_{l=0}^{n-1}\Vert -\nabla \xi_{j}^{l}+\eta_{j}^{l+1}\times B \Vert^2\nonumber\\
&\qquad+N\Delta t\sum_{l=0}^{n-1}\Vert -\nabla \xi_{j}^{l+1}+\eta_{j}^{l}\times B \Vert^2+N\Delta t\sum_{l=0}^{n-1}\Vert \nabla \xi_{j}^{l+1}\Vert^2 
+\frac{7N}{12M^2}\Delta t
\Vert\nabla \eta_{j}^{n}\Vert^{2}.\nonumber
\end{align}

\noindent Using the triangle
inequality on the error equation gives

\begin{align}\label{ineq:errlast2}
&\Vert e_{u,j}^{n}\Vert^{2}+\frac{1}{2}\sum_{l=0}^{n-1}\|e_{u,j}^{l+1}-e_{u,j}^{l}\|^{2}+\Delta t\sum_{l=0}^{n-1}
\frac{N}{2M^2}\Vert\nabla e_{u,j}^{l+1}\Vert^{2}\\
&\qquad +N\Delta t\Vert e_{u,j}^{n}\times B\Vert^2+N\Delta t\Vert \nabla e_{\phi, j}^{n}\Vert^2 
+N\Delta t\sum_{l=0}^{n-1}\Vert -\nabla e_{\phi, j}^{l}+e_{u,j}^{l+1}\times B \Vert^2\nonumber\\
&\qquad+N\Delta t\sum_{l=0}^{n-1}\Vert -\nabla e_{\phi, j}^{l+1}+e_{u,j}^{l}\times B \Vert^2+N\Delta t\sum_{l=0}^{n-1}\Vert \nabla e_{\phi, j}^{l+1}\Vert^2 
+\frac{7N}{12M^2}\Delta t
\Vert\nabla e_{u, j}^{n}\Vert^{2}\nonumber\\
&\leq e^{\frac{T\tilde{C}}{1-\Delta t \tilde{C}}}
\bigg\{\Delta t\left(\frac{CM^6}{N^3 }+12N \Vert B\Vert_{L^{\infty}}^2\right)\| e^{0}_{u,j} \|^{2}+  \Vert e
_{u, j}^{0}\Vert^{2}+N\Delta t\Vert e_{u, j}^{0}\times B\Vert^2\nonumber\\
&+N\Delta t\Vert \nabla e_{\phi, j}^{0}\Vert^2+\frac{N}{2M^2}\Delta t
\Vert\nabla e_{u,j}^{0}\Vert^{2}+\Delta t\left(\frac{CM^6}{N^3 }+12N \Vert B\Vert_{L^{\infty}}^2\right)\| \eta^{0}_{j} \|^{2}\nonumber\\
&+  \Vert \eta
_{j}^{0}\Vert^{2}+N\Delta t\Vert \eta_{j}^{0}\times B\Vert^2+N\Delta t\Vert \nabla \xi_{j}^{0}\Vert^2+\frac{7N}{12M^2}\Delta t
\Vert\nabla \eta_{j}^{0}\Vert^{2}\nonumber\\
&+C\frac{N}{M^2} h^{2k}\vertiii{ u_j }^2_{2,k+1}+C h^{2k+1}\Delta t^{-1}\vertiii{ u_j }^2_{2,k+1}+ Ch\Delta t\vertiii{ \nabla u_{j,t}}^2_{2,0}
\nonumber\\
&+C\frac{M^2}{N} h^{2k}\vertiii{ u_j }^2_{2,k+1}+C\frac{M^2}{N} h^{2k}\vertiii{ u_j}^4_{4, k+1}+C\frac{M^3}{N} h^{2k}+C\frac{M^2}{N}\Delta t^2\vertiii{ \nabla u_{j,t}}^2_{2,0}\nonumber\\
&+CM^2N h^{2s+2}\vertiii{p_{j}}_{2,s+1}^{2}+C\frac{M^2}{N}h^{2k+2}\vertiii{u_{j,t}}_{2,k+1}^{2} +C\frac{M^2}{N}\Delta t^{2}\vertiii{u_{j,tt}}_{2,0}^{2}\nonumber\\
& +C N h^{2k}\vertiii{ \phi_j }^2_{2,k+1}+CN\Vert B\Vert_{L^{\infty}}^2 h^{2k+2}\vertiii{ u_j }^2_{2,k+1}+ CN\Delta t^2\vertiii{ \nabla \phi_{j,t}}^2_{2,0}\nonumber\\
&+C Nh^{2k+2}\vertiii{ \phi_j }^2_{2,k+1}+CN\Vert B\Vert_{L^{\infty}}^2\Delta t^2\vertiii{ u_{j,t}}^2_{2,0}
\bigg\}+\Vert \eta_{j}^{n}\Vert^{2}+\frac{1}{2}\sum_{l=0}^{n-1}\|\eta_{j}^{l+1}-\eta_{j}^{l}\|^{2}\nonumber\\
&+\Delta t\sum_{l=0}^{n-1}
\frac{N}{2M^2}\Vert\nabla \eta_{j}^{l+1}\Vert^{2}+N\Delta t\Vert \eta_{j}^{n}\times B\Vert^2+N\Delta t\Vert \nabla \xi_{j}^{n}\Vert^2 \nonumber\\
&+N\Delta t\sum_{l=0}^{n-1}\Vert -\nabla \xi_{j}^{l}+\eta_{j}^{l+1}\times B \Vert^2+N\Delta t\sum_{l=0}^{n-1}\Vert -\nabla \xi_{j}^{l+1}+\eta_{j}^{l}\times B \Vert^2\nonumber\\
&+N\Delta t\sum_{l=0}^{n-1}\Vert \nabla \xi_{j}^{l+1}\Vert^2 
+\frac{N}{2M}\Delta t
\Vert\nabla \eta_{j}^{n}\Vert^{2}.\nonumber
\end{align}

\noindent Applying the interpolation inequalities \eqref{Interp1}, \eqref{interp2}, and \eqref{interp3} and absorbing constants into a new constant $C$ yields

\begin{align}\label{ineq:errlast3}
&\Vert e_{u,j}^{n}\Vert^{2}+\frac{1}{2}\sum_{l=0}^{n-1}\|e_{u,j}^{l+1}-e_{u,j}^{l}\|^{2}+\Delta t\sum_{l=0}^{n-1}
\frac{N}{2M^2}\Vert\nabla e_{u,j}^{l+1}\Vert^{2}\\
&\qquad +N\Delta t\Vert e_{u,j}^{n}\times B\Vert^2+N\Delta t\Vert \nabla e_{\phi, j}^{n}\Vert^2 
+N\Delta t\sum_{l=0}^{n-1}\Vert -\nabla e_{\phi, j}^{l}+e_{u,j}^{l+1}\times B \Vert^2\nonumber\\
&\qquad+N\Delta t\sum_{l=0}^{n-1}\Vert -\nabla e_{\phi, j}^{l+1}+e_{u,j}^{l}\times B \Vert^2+N\Delta t\sum_{l=0}^{n-1}\Vert \nabla e_{\phi, j}^{l+1}\Vert^2 
+\frac{7N}{12M^2}\Delta t
\Vert\nabla e_{u, j}^{n}\Vert^{2}\nonumber\\
&\leq e^{\frac{T\tilde{C}}{1-\Delta t \tilde{C}}}
\bigg\{\Delta t\left(\frac{CM^6}{N^3 }+13N \Vert B\Vert_{L^{\infty}}^2\right)\| e^{0}_{u,j} \|^{2}+  \Vert e
_{u, j}^{0}\Vert^{2}+N\Delta t\Vert \nabla e_{\phi, j}^{0}\Vert^2\nonumber\\
&+\frac{7N}{12M^2}\Delta t
\Vert\nabla e_{u,j}^{0}\Vert^{2}+C\left(\frac{CM^6}{N^3 }+12N \Vert B\Vert_{L^{\infty}}^2\right)h^{2k+2}\Delta t\| u^{0}_{j} \|^{2}_{k+1}+  Ch^{2k+2}\Vert u
_{j}^{0}\Vert^{2}_{k+1}\nonumber\\
&+CN\Vert B\Vert^2_{\infty}h^{2k+2}\Delta t\Vert u_{j}^{0}\Vert_{k+1}^2+CNh^{m}\Delta t\Vert \phi_{j}^{0}\Vert_{m+1}^2+C\frac{N}{2M^2}h^{2k}\Delta t
\Vert u_{j}^{0}\Vert_{k+1}^{2}\nonumber\\
&+C\frac{N}{M^2} h^{2k}\vertiii{ u_j }^2_{2,k+1}+C h^{2k+1}\Delta t^{-1}\vertiii{ u_j }^2_{2,k+1}+ Ch\Delta t\vertiii{ \nabla u_{j,t}}^2_{2,0}
\nonumber\\
&+C\frac{M^2}{N} h^{2k}\vertiii{ u_j }^2_{2,k+1}+C\frac{M^2}{N} h^{2k}\vertiii{ u_j}^4_{4, k+1}+C\frac{M^3}{N} h^{2k}+C\frac{M^2}{N}\Delta t^2\vertiii{ \nabla u_{j,t}}^2_{2,0}\nonumber\\
&+CM^2N h^{2s+2}\Vert|p_{j}|\Vert_{2,s+1}^{2}+C\frac{M^2}{N}h^{2k+2}\Vert
|u_{j,t}|\Vert_{2,k+1}^{2} +C\frac{M^2}{N}\Delta t^{2}\Vert|u_{j,tt}|\Vert_{2,0}^{2}\nonumber\\
& +C N h^{2m}\vertiii{ \phi_j }^2_{2,m+1}+CN\Vert B\Vert_{L^{\infty}}^2 h^{2k+2}\vertiii{ u_j }^2_{2,k+1}+ CN\Delta t^2\vertiii{ \nabla \phi_{j,t}}^2_{2,0}\nonumber\\
&+C Nh^{2m+2}\vertiii{ \phi_j }^2_{2,m+1}+CN\Vert B\Vert_{L^{\infty}}^2\Delta t^2\vertiii{ u_{j,t}}^2_{2,0}
\bigg\}+Ch^{2k+2}\vertiii{u_j}^{2}_{\infty, k+1}\nonumber\\
&+Ch^{2k+2}\Delta t \vertiii{u_{j,t}}_{2, k+1}+C
\frac{N}{M^2}h^{2k}\vertiii{u_{j}}_{2, k+1}+CN\Vert B\Vert_{L^{\infty}}^2 h^{2k+2}\Delta t\vertiii{ u_{j}}_{\infty, k+1}\nonumber\\
&+CNh^{2m}\Delta t\vertiii{  \phi_{j}}_{\infty, m+1} +CNh^{2m} \vertiii{\phi_j}_{2,m+1}+CN\Vert B\Vert_{L^{\infty}}^2h^{2k+2}\vertiii{u_j}_{2, k+1}\nonumber\\
&+CNh^{2m}\vertiii{ \phi_{j}}_{2, m+1} 
+C\frac{N}{M}h^{2k}\Delta t
\vertiii{ u_{j}}_{\infty, k+1}.\nonumber
\end{align}

This completes the proof of Theorem  \ref{th:errBEFE-Ensemble}.
\end{proof}
}

\section{Numerical Experiments}

In this section we present numerical experiments for Algorithm \ref{Algo} demonstrating the convergence and stability theorems proven in the previous sections. For all examples we will use the finite element triplet ($P^{2}$--$P^{1}$--$P^{2 }$) and the finite element software package FEniCS \cite{FENICS}.

\subsection{Convergence Test}
For our first test problem we verify the convergence rates proven in section \ref{err_analysis} using a variation of the test problem used in \cite{Layton2014}. Take the time interval $0 \leq t \leq 1$, M = 16, N = 20, $\Omega = [0,\pi]^{2}$, and the imposed magnetic field $B = (0,0,1)$. 
We consider the true solution $(u,p,\phi)$ given by
\begin{equation}
\left\{\begin{aligned}
&u_{\epsilon} = (5\cos(5x)\sin(5y),-5\sin(5x)\cos(5y),0)(1 + \epsilon)e^{-5t}, \\
&p = 0,\\
&\phi_{\epsilon} = (\cos(5x)\cos(5y) + x^{2} - y^{2})(1 + \epsilon)e^{-5t},
\end{aligned}\right.
\end{equation}
where $\epsilon$ is a given perturbation. For this problem we will consider two perturbations $\epsilon_{1} = 10^{-3}$ and $\epsilon_{2} = -10^{-3}$. The boundary conditions are taken to be $u_{h} = u_{\epsilon}$ and $\phi_{h} = \phi_{\epsilon}$ on $\partial\Omega$. The initial conditions and source terms are chosen to correspond with the exact solution for the given perturbation. As can be seen in tables \ref{exper1_table1} \ref{exper1_table2} \ref{exper1_table3} and \ref{exper1_table4} we achieve the expected convergence rates.

\begin{table}
	\begin{center}
	\begin{tabular}{  l c c c c c   }
		\hline
		$h$ & $\Delta t$ & $\|u_{1} - u_{1,h} \|_{\infty,0}$ & Rate & $\|\nabla u_{1} - \nabla u_{1,h} \|_{2,0}$ & Rate    \\ \hline
		1/20 & 1/160 & 8.323e-1 & -  & 4.847e+0 & -    \\ 
		1/40 & 1/320 & 5.141e-1 & 0.695 & 2.787e+0 & 0.798  \\
		1/60 & 1/480 & 3.615e-1 & 0.869 & 1.942e+0 & 0.891   \\
		1/80 & 1/640 & 2.779e-1 & 0.915 & 1.489e+0 & 0.924   \\
		1/120 & 1/960 & 1.895e-1 & 0.945 &1.014e+0 & 0.946   \\
	\end{tabular}
	\caption{Error and convergence rates for the first ensemble member in $u_{h}$ and $\nabla u_{h}$ }
	\label{exper1_table1}
	\end{center}
	\end{table}

\begin{table}
\begin{center}
	\begin{tabular}{  l c c c c c   }
		\hline
		$h$ & $\Delta t$ & $\|\phi_{1} - \phi_{1,h} \|_{\infty,0}$ & Rate & $\|\nabla \phi_{1} - \nabla \phi_{1,h} \|_{2,0}$ & Rate    \\ \hline
		1/20 & 1/160 & 1.358e-1 & -  & 7.188e-1 & -    \\ 
		1/40 & 1/320 & 8.451e-2 & 0.684 & 4.250e-1 & 0.758  \\
		1/60 & 1/480 & 5.957e-2 & 0.862 & 2.962e-1 & 0.891   \\
		1/80 & 1/640 & 4.587e-2 & 0.909 & 2.269e-1 & 0.927   \\
		1/120 & 1/960 & 3.135e-2 & 0.939 &1.543e-1 & 0.951   \\
	\end{tabular}
	\caption{Error and convergence rates for the first ensemble member in $\phi_{h}$ and $\nabla \phi_{h}$ }
	\label{exper1_table2}
\end{center}
\end{table}

\begin{table}
	\begin{center}
	\begin{tabular}{  l c c c c c   }
		\hline
		$h$ & $\Delta t$ & $\|u_{2} - u_{2,h} \|_{\infty,0}$ & Rate & $\|\nabla u_{2} - \nabla u_{2,h} \|_{2,0}$ & Rate    \\ \hline
		1/20 & 1/160 & 8.305e-1 & -  & 4.836e+0 & -    \\ 
		1/40 & 1/320 & 5.130e-1 & 0.695 & 2.781e+0 & 0.798  \\
		1/60 & 1/480 & 3.606e-1 & 0.869 & 1.938e+0 & 0.891   \\
		1/80 & 1/640 & 2.772e-1 & 0.915 & 1.486e+0 & 0.924   \\
		1/120 & 1/960 & 1.890e-1 & 0.944 &1.012e+0 & 0.946   \\
	\end{tabular}
	\caption{Error and convergence rates for the second ensemble member in $u_{h}$ and $\nabla u_{h}$ }
	\label{exper1_table3}
	\end{center}
	\end{table}

\begin{table}
\begin{center}
	\begin{tabular}{  l c c c c c   }
		\hline
		$h$ & $\Delta t$ & $\|\phi_{2} - \phi_{2,h} \|_{\infty,0}$ & Rate & $\|\nabla \phi_{2} - \nabla \phi_{2,h} \|_{2,0}$ & Rate    \\ \hline
		1/20 & 1/160 & 1.355e-1 & -  & 7.173e-1 & -    \\ 
		1/40 & 1/320 & 8.431e-2 & 0.684 & 4.241e-1 & 0.758  \\
		1/60 & 1/480 & 5.944e-2 & 0.862 & 2.956e-1 & 0.891   \\
		1/80 & 1/640 & 4.576e-2 & 0.909 & 2.264e-1 & 0.927   \\
		1/120 & 1/960 & 3.127e-2 & 0.939 &1.540e-1 & 0.951   \\
	\end{tabular}
	\caption{Error and convergence rates for the second ensemble member in $\phi_{h}$ and $\nabla \phi_{h}$ }
	\label{exper1_table4}
\end{center}
\end{table}

\subsection{Efficiency Test}
For our second experiment we will consider the same setting as the first numerical experiment except we will use $11$ perturbations $\epsilon_{i} = 10^{-2} - .0009*i, i = 0,\ldots,10$. In order to measure the efficiency of the ensemble method we compare the CPU time measured in seconds and accuracy of Algorithm \ref{Algo} versus the non-ensemble IMEX version of Algorithm \ref{Algo} in terms of the averages $\bar{u}^n$ and $\bar{\phi}^{n}$. For both algorithms we will use the direct LU solver MUMPS \cite{MUMPS1} \cite{MUMPS2}. We see in tables \ref{en_effic} and \ref{ser_effic} that the ensemble algorithm is able to achieve similar accuracy to the non-ensemble algorithm with significant cost savings.
\begin{table}
	\begin{center}
		\begin{tabular}{  l c c c c    }
			\hline
			$h$ & $\Delta t$ & $\|\bar{u} - \bar{u}_{en,h} \|_{\infty,0}$ & $\|\bar{\phi} - \bar{\phi}_{en,h} \|_{\infty,0}$ & CPU time    \\ \hline
			1/20 & 1/160 & 8.355e-1 & 1.362e-1  & 1.3134e+2  \\ 
			1/40 & 1/320 & 5.157e-1 & 8.477e-2 & 9.1382e+2   \\
			1/60 & 1/480 & 3.623e-1 &5.972e-2 & 3.3533e+3    \\
			1/80 & 1/640 & 2.783e-1 & 4.595e-2 & 8.2587e+3   \\
			1/120 & 1/960 & 1.894e-1 & 3.137e-2 &2.6761e+4   \\
		\end{tabular}
		\caption{Error and CPU time for computing $\bar{u}_{h}$ and $\bar{\phi}_{h}$ with Algorithm \ref{Algo}}
		\label{en_effic}
	\end{center}
\end{table}    

\begin{table}
	\begin{center}
		\begin{tabular}{  l c c c c    }
			\hline
			$h$ & $\Delta t$ & $\|\bar{u} - \bar{u}_{ser,h} \|_{\infty,0}$ & $\|\bar{\phi} - \bar{\phi}_{ser,h} \|_{\infty,0}$ & CPU time    \\ \hline
			1/20 & 1/160 & 8.355e-1 & 1.362e-1  & 2.5471e+2  \\ 
			1/40 & 1/320 & 5.157e-1 & 8.477e-2 & 1.6840e+3   \\
			1/60 & 1/480 & 3.623e-1 & 5.972e-1 & 8.1010e+3    \\
			1/80 & 1/640 & 2.783e-1 & 4.595e-2 & 2.0536e+4   \\
			1/120 & 1/960 & 1.894e-1 & 3.137e-2 &4.3062e+4   \\
		\end{tabular}
		\caption{Error and CPU time for computing $\bar{u}_{h}$ and $\bar{\phi}_{h}$ serially with the non-ensemble method. }
		\label{ser_effic}
	\end{center}
\end{table}    

\subsection{Stability Test}

In this experiment we test the time step restriction for the stability of our algorithm by using a variation on the test for liquid aluminum performed in \cite{LTT14}. Let $0 \leq t \leq 1$, M = 12255, N = 347, $\Omega = [0,10^{-1}]^{2}$, and the imposed magnetic field $B = (0,0,1)$. We take $f$ and the boundary conditions equal to $0$ and the initial conditions to be equal to 
\begin{equation*}
\begin{aligned}
u_{0}(x,y,\epsilon) &= (10\pi \cos(10 \pi x) \sin(10 \pi y)  , -10 \pi \sin(10 \pi x) cos(10 \pi y),0) (1 + \epsilon), \\
\phi_{0}(x,y, \epsilon) &= ( \cos(10 \pi x) \cos(10 \pi y) + x^{2} - y^{2}) (1 + \epsilon),
\end{aligned}
\end{equation*}
for which we will consider the two perturbations $\epsilon_{1} = 10^{-1}$ and $\epsilon_{2} = 10^{-2}$. Due to the fact that there is no external energy exchange or body forces the energy in the system should decay to $0$ over time assuming the algorithm is stable. For $h = \frac{1}{10}$ we compute the average energy $E^{n} = \frac{1}{2} \|\bar{\phi}^{n}\|^{2} + \frac{1}{2} \|\bar{u}^{n}\|^{2}$ over a number of different time steps. As we can see in figure \ref{exper3_fig1} our method is unstable for $\Delta t = \frac{1}{10}, \frac{1}{100}$, but becomes stable with $\Delta t = \frac{1}{1000}$ .

\begin{figure}[h]
\centering
\includegraphics[height=8cm]{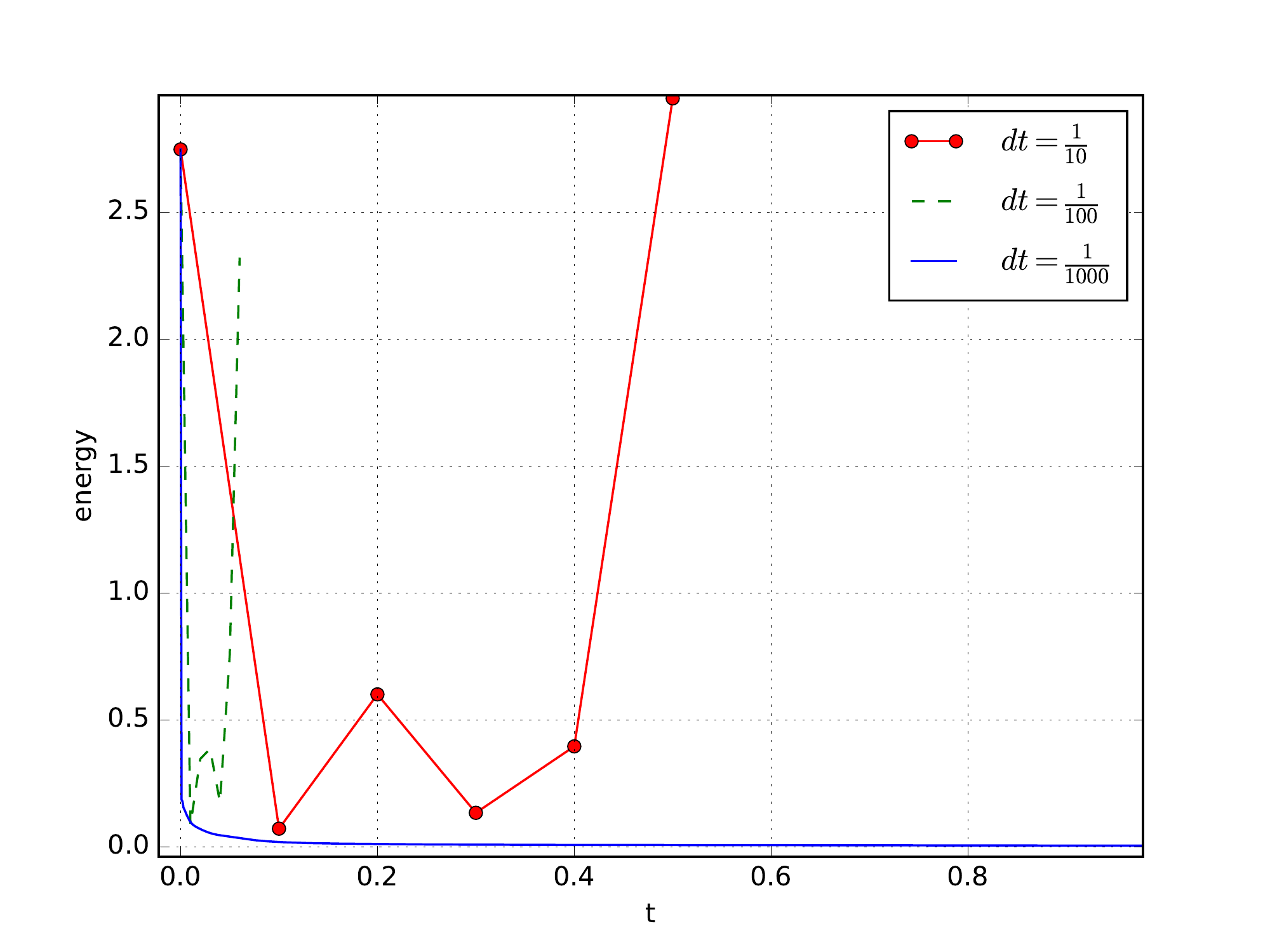} 
\vspace{-.1in}
\caption{The decay of the system energy for algorithm \eqref{Algo} with several different time steps.}
\label{exper3_fig1}
\end{figure}

\end{document}